\documentclass[11pt]{article}

\usepackage{amssymb,amsmath,amsfonts,amsthm}
\usepackage[shortlabels]{enumitem}
\usepackage{latexsym}
\usepackage{graphics}
\usepackage{indentfirst}
\usepackage{tikz}
\usepackage{mathtools}
\usepackage{hyperref}
\usepackage{multicol}
\usepackage[capitalise, noabbrev, nameinlink]{cleveref}
\allowdisplaybreaks

\newtheorem{innercustomthm}{Theorem}
\newenvironment{customthm}[1]
  {\renewcommand\theinnercustomthm{#1}\innercustomthm}
  {\endinnercustomthm}

\newtheorem*{thm*}{Theorem}
\newtheorem{thm}{Theorem}
\newtheorem{lem}[thm]{Lemma}

\newtheorem{obs}[thm]{Observation}
\newtheorem{cor}[thm]{Corollary}

\newtheorem{cl}[thm]{Claim}
\newtheorem{ques}[thm]{Question}
\newtheorem{define}[thm]{Definition}

\newtheorem{rem}[thm]{Remark}

\newcommand{\N}{\mathbb{N}}
\newcommand{\Z}{\mathbb{Z}}

\newcommand{\R}{\mathbb{R}}

\newcommand{\col}{\mathrm{col}}
\newcommand{\Sh}{\textsc{Shade}}
\newcommand{\ShSa}{\textsc{ShadeSave}}
\newcommand{\Sa}{\textsc{Save}}
\newcommand{\DSa}{\textsc{DelSave}}

\begin{document}

\title{Fractional Strict Degeneracy of Graphs}

\author{Daniel Dominik$^1$ and Jeffrey A. Mudrock$^2$}

\footnotetext[1]{Department of Mathematics, Thornton Township High School, Harvey, IL 60426.  E-mail:  {\tt {ddominik@hawk.iit.edu}}}
\footnotetext[2]{Department of Mathematics and Statistics, University of South Alabama, Mobile, AL 36688. E-mail: {\tt mudrock@southalabama.edu}}

\maketitle

\begin{abstract}

DP-coloring (also called correspondence coloring) is a generalization of list coloring introduced by  Dvo\v{r}\'{a}k and Postle in 2015.  The DP-chromatic number of a graph $G$, $\chi_{_{DP}}(G)$, is the analogue of the chromatic number of $G$ in the DP context and is bounded above by the degeneracy of $G$ plus one.  Over the last two years a plethora of authors have introduced variations on the notion of degeneracy and used these new ideas to give improved bounds on the DP-chromatic number of certain families of graphs. 

Fractional DP-coloring is a generalization of fractional list coloring introduced by Bernshteyn, Kostochka, and Zhu in 2019.  In this paper we introduce two analogues of the degeneracy of a graph to the fractional context, each of which bound its fractional DP-chromatic number from above.  We use these analogues to bound the fractional DP-chromatic number of a variety of graphs including unicyclic graphs, some complete bipartite graphs, and sparse graphs.

\medskip



\end{abstract}

\section{Introduction}\label{intro}

In this paper all graphs are nonempty, finite, simple graphs unless otherwise noted.  Generally speaking we follow West~\cite{W01} for terminology and notation.  The set of natural numbers is $\N = \{1,2,3, \ldots \}$.  For $m \in \N$, we write $[m]$ for the set $\{1, \ldots, m \}$, and we adopt the convention that $[0] = \emptyset$.  For $a,b\in\Z$, we write $[a:b]$ for the set $\{x\in\Z : a\leq x\leq b\}$.  We adopt the convention that $\sum_{x\in\emptyset}x=0$.  If $G$ is a graph and $S, U \subseteq V(G)$, we use $G[S]$ for the subgraph of $G$ induced by $S$, and we use $E_G(S, U)$ for the subset of $E(G)$ with one endpoint in $S$ and the other endpoint in $U$.  For graph $G$ and $v\in V(G)$, $G-v=G[V(G) \setminus \{v\}]$; likewise, for  $S\subseteq V(G)$, $G-S=G[V(G)\setminus S]$.  For graphs $G$ and $H$ we say \emph{$G$ is a subdivision of $H$} if $G$ can be obtained by replacing the edges of $H$ with pairwise internally disjoint paths each of which have length at least one.  For $v \in V(G)$, we write $d_G(v)$ for the degree of vertex $v$ in the graph $G$, and we write $N_G(v)$ for the neighborhood of vertex $v$ in the graph $G$.  Also, for $S \subseteq V(G)$, we let $N_G(S) = \bigcup_{v \in S} N_G(v)$.  We use $K_{n,m}$ to denote the complete bipartite graphs with partite sets of size $n$ and $m$.  We use $K_{n_1,\ldots,n_k}$ to denote the complete $k$-partite graphs with partite sets of size $n_1$,$\ldots$, and $n_k$.  For $k \in \N$ and $G=P_k$, when we say that \emph{the vertices of $G$ in order are $v_1,\ldots,v_k$} we mean that two vertices are adjacent in $G$ if and only if they appear consecutively in this ordering.  For $k\geq3$ and $G=C_k$, when we say that \emph{the vertices of $G$ in cyclic order are $v_1,\ldots,v_k$} we mean that two vertices are adjacent in $G$ if and only if they appear consecutively in this ordering or if they are $v_1$ and $v_k$.  We say $G$ is \emph{unicyclic} if $G$ is connected and contains exactly one cycle.  

\subsection{DP-Coloring and Fractional DP-Coloring}

In this paper we present two extensions of the notion of degeneracy that bound the fractional DP-chromatic number of a graph, and consequently its fractional chromatic number, from above.  We begin by presenting some classical notions and briefly reviewing fractional DP-coloring.  Given a graph $G$, in the classical vertex coloring problem we wish to color the elements of $V(G)$ with colors from the set $[k]$ so that adjacent vertices receive different colors, a so-called \emph{proper $k$-coloring}.  We say $G$ is \emph{$k$-colorable} when a proper $k$-coloring of $G$ exists.  The \emph{chromatic number} of $G$, denoted $\chi(G)$, is the smallest $k$ such that $G$ is $k$-colorable.

A \emph{set coloring} of a graph $G$ is a function that assigns a set to each vertex of $G$ such that the sets assigned to adjacent vertices are disjoint.  For $a, b \in \N$ with $a \geq b$, an \emph{$(a,b)$-coloring} of graph $G$ is a set coloring $f$ of $G$ such that the codomain of $f$ is the set of $b$-element subsets of $[a]$.  We say that $G$ is \emph{$(a,b)$-colorable} when an $(a,b)$-coloring of $G$ exists.  The \emph{fractional chromatic number} of $G$, denoted $\chi^*(G)$, is defined by $\chi^*(G) = \inf \{ a/b : \text{$G$ is $(a,b)$-colorable} \}.$  Since any graph $G$ is $(\chi(G),1)$-colorable, we have that $\chi^*(G) \leq \chi(G)$.  Moreover, if graph $G$ has an edge, it is not $(a,b)$-colorable whenever $a<2b$ which means $2\leq \chi^*(G)$.   It is also well known that the infimum in the definition of $\chi^*(G)$ is actually a minimum~\cite{SU97}.

List coloring is a well studied variation on classical vertex coloring that was introduced in the 1970s~\cite{ERT, V}.  For graph $G$ we say that a function $L:V(G)\to 2^\N$ is a \emph{list assignment of $G$}.  A \emph{proper $L$-coloring} of $G$ is a proper coloring of $G$ that assigns to each $v\in V(G)$ a color from $L(v)$.  If $k\in\N$ and $L$ is a list assignment such that $|L(v)|=k$ for each $v\in V(G)$, we say \emph{$L$ is a $k$-assignment of $G$}. We say $G$ is \emph{$k$-choosable} if for each $L$ that is a $k$-assignment of $G$, there exists a proper $L$-coloring of $G$. 
The \emph{list chromatic number} of $G$, denoted $\chi_\ell(G)$, is the smallest $k$ for which $G$ is $k$-choosable.  Since a $k$-assignment can assign the same $k$ colors to every vertex of a graph, $\chi(G) \leq \chi_\ell(G)$.

Fractional coloring can be extended to the list context in a natural way.  Given an $a$-assignment $L$ for graph $G$ and $b \in \N$ such that $a \geq b$, we say that $f$ is an \emph{$(L,b)$-coloring} of $G$ if $f$ is a set coloring of $G$ such that for each each $v \in V(G)$, $f(v) \subseteq L(v)$ and $|f(v)| = b$.  We say that $G$ is \emph{$(L,b)$-colorable} if an $(L,b)$-coloring of $G$ exists.  Also, for $a, b \in \N$ with $a \geq b$, graph $G$ is \emph{$(a,b)$-choosable} if $G$ is $(L,b)$-colorable whenever $L$ is an $a$-assignment for $G$.  The \emph{fractional list chromatic number} of $G$, denoted $\chi_\ell^*(G)$, is defined by $\chi_\ell^*(G) = \inf \{ a/b : \text{$G$ is $(a,b)$-choosable} \}.$  Since an $(a,b)$-choosable graph is $(a,b)$-colorable, $\chi^*(G) \leq \chi_\ell^*(G)$.  In 1997, Alon, Tuza, and Voigt~\cite{AT97} famously proved that for any graph $G$, $\chi_\ell^*(G) = \chi^*(G)$.  Moreover, they showed that for any graph $G$, there is an $M \in \N$ such that $G$ is $(M, M/\chi^*(G))$-choosable.  So, the infimum in the definition of $\chi_\ell^*(G)$ is actually a minimum.

In~\cite{ERT} Erd\H{o}s et al. asked: If $G$ is $(a,b)$-choosable, does it follow that $G$ is $(at,bt)$-choosable for each $t \in \N$?  Tuza and Voigt~\cite{TV96} showed that the answer to this question is yes when $a=2$ and $b=1$.  However, in~\cite{DH18}, a graph that is $4$-choosable but not $(8,2)$-choosable is constructed.  Interestingly, for the notion we develop in this paper the answer to the analogue of this question is yes (see Theorem~\ref{thm:ks,kt-degenerate}).

In 2015, Dvo\v{r}\'{a}k and Postle~\cite{DP} introduced DP-coloring (they called it correspondence coloring).  Intuitively, DP-coloring is a generalization of list coloring where each vertex in the graph still gets a list of colors but identification of which colors are different can vary from edge to edge.  We now give the formal definition. Suppose $G$ is a graph.  A \emph{cover} of $G$ is a pair $\mathcal{H} = (L,H)$ consisting of a graph $H$ and a function $L: V(G) \rightarrow 2^{V(H)}$ satisfying the following four conditions:

\vspace{5mm}

\noindent(1) the set $\{L(u) : u \in V(G) \}$ is a partition of $V(H)$ with $|V(G)|$ parts; \\
(2) for every $u \in V(G)$, $L(u)$ is an independent set of vertices in $H$; \\
(3) if $E_H(L(u),L(v))$ is nonempty, then $uv \in E(G)$; \\
(4) if $uv \in E(G)$, then $E_H(L(u),L(v))$ is a matching (the matching may be empty).

\vspace{5mm}

Suppose $\mathcal{H} = (L,H)$ is a cover of $G$.  We say $\mathcal{H}$ is \emph{$k$-fold} if $|L(u)|=k$ for each $u \in V(G)$.  An \emph{$\mathcal{H}$-coloring} of $G$ is an independent set $I\subseteq V(H)$ such that $|I \cap L(u)|=1$ for each $u \in V(G)$.  The \emph{DP-chromatic number} of a graph $G$, denoted $\chi_{_{DP}}(G)$, is the smallest $k \in \N$ such that $G$ admits an $\mathcal{H}$-coloring for every $k$-fold cover $\mathcal{H}$ of $G$.  Since for any $k$-assignment $L$ for $G$, there exists a $k$-fold cover $\mathcal{H}$ of $G$ such that $G$ is $L$-colorable if and only if it is $\mathcal{H}$-colorable~\cite{DP}, we know that $\chi(G) \leq \chi_\ell(G) \leq \chi_{_{DP}}(G).$

We now turn our attention to fractional DP-coloring.  Suppose $\mathcal{H} = (L,H)$ is an $a$-fold cover of $G$ and $b \in \N$ satisfies $a \geq b$.  Then, $G$ is \emph{$(\mathcal{H},b)$-colorable} if there is an independent set $S \subseteq V(H)$ such that $|S \cap L(v)| \geq b$ for each $v \in V(G)$.  Equivalently, one could require $|S \cap L(v)| = b$ for each $v \in V(G)$; in this case we call $S$ an \emph{independent $b$-fold transversal of $\mathcal{H}$}. We refer to $S$ as an \emph{$(\mathcal{H},b)$-coloring} of $G$.  Note that an $\mathcal{H}$-coloring of $G$ is also an $(\mathcal{H},1)$-coloring of $G$.  For $a,b \in \N$ and $a \geq b$, we say $G$ is \emph{$(a,b)$-DP-colorable} if for any $a$-fold cover of $G$, $\mathcal{H}$, $G$ is $(\mathcal{H},b)$-colorable.  The \emph{fractional DP-chromatic number} of $G$, denoted $\chi_{_{DP}}^*(G)$, is defined by $\chi_{_{DP}}^*(G) = \inf \{ a/b : \text{$G$ is $(a,b)$-DP-colorable} \}.$
It's easy to show that an $(a,b)$-DP-colorable graph is $(a,b)$-choosable, and any graph $G$ is $(\chi_{_{DP}}(G),1)$-DP-colorable.  So,
$$\chi^*(G) = \chi_{\ell}^*(G) \leq \chi_{_{DP}}^*(G) \leq \chi_{_{DP}}(G).$$ 
 In~\cite{BKZ} the following result is proven.

\begin{thm} [\cite{BKZ}] \label{thm: fracDP2}
Let $G$ be a connected graph.  Then, $\chi_{_{DP}}^*(G) \leq 2$ if and only if $G$ contains no odd cycles and at most one even cycle.  Furthermore, if $G$ contains no odd cycles and exactly one even cycle, then $\chi_{_{DP}}^*(G)=2$ even though $2$ is not contained in the set $\{ a/b : \text{$G$ is $(a,b)$-DP-colorable} \}$.
\end{thm}

So, unlike the fractional chromatic number and fractional list chromatic number, the infimum in the definition of the fractional DP-chromatic number is not always a minimum.  In~\cite{BKZ} it is shown that $\chi^*_{_{DP}}(G)\geq d/(2\ln d)$, whenever $d \geq 4$ and $d$ is the maximum average degree of $G$, which means there are complete bipartite graphs with fractional DP-chromatic number arbitrarily higher than $\chi^*(G) = 2$.  Thus, in general, the fractional DP-chromatic number of a graph cannot be bounded above by a function its fractional chromatic number.  In 1997, Alon, Tuza, and Voigt~\cite{AT97} showed that $C_{2r+1}$ is $(2r+1,r)$-choosable, and this was recently extended to the DP context.

\begin{thm} [\cite{DKM25}] \label{thm: oddcycle}
$C_{2r+1}$ is $(2r+1,r)$-DP-colorable.  Consequently, $2+1/r = \chi^*(C_{2r+1})= \chi_{_{\ell}}^*(C_{2r+1}) = \chi_{_{DP}}^*(C_{2r+1})$.
\end{thm}

Interestingly the DP analogue of the aforementioned $(a,b)$-choosability question in~\cite{ERT} remains open.

\begin{ques} \label{ques: Erdos2}
If $G$ is $(a,b)$-DP-colorable, does it follow that $G$ is $(at,bt)$-DP-colorable for each $t \in \N$?
\end{ques}

In~\cite{DKM25} a probabilistic argument is used to prove the following.

\begin{thm} [\cite{DKM25}] \label{thm:completeMultipartite}
	Suppose $G=K_{n_1,n_2,\ldots,n_m}$.  Let $p^*_m=1$ and $p^*_j$ be the solution in $(0,1)$ to $p=p^*_{j+1}(1-p)^{n_j}$ for all $j\in[m-1]$.  Then $\chi^*_{_{DP}}(G)=\chi^*_{_{DP}}(K_{n_1,n_2,\ldots,n_m})\leq\frac{1}{p^*_1}$.
\end{thm}
Theorem~\ref{thm:completeMultipartite} gives the best known upper bound of which we are aware on $\chi^*_{_{DP}}(K_{n,m})$ for all $n\geq2$ and sufficiently large $m$. In fact, when $n=2$ it provides the best known upper bound for each $m \geq 3$: $\chi^*_{_{DP}}(K_{2,m})\leq2/(3-\sqrt{5})\approx2.619$.  In this paper we improve this bound for $K_{2,3}$, $K_{2,4}$, and $K_{2,5}$.

\subsection{Strict Degeneracy}

For each $k \in \N$ the \emph{$k$-core} of a graph $G$ is the subgraph of $G$ obtained by successively deleting vertices of degree less than $k$; equivalently, the $k$-core of $G$ is the maximal subgraph $H$ of $G$ satisfying $d_H(v) \geq k$ for each $v \in V(H)$.  For example, the 2-core of a forest is the empty graph, the 2-core of a unicyclic graph is the cycle it contains, and the 3-core of a cycle is the empty graph.   

Recall that a graph $G$ is said to be $d$-degenerate if there exists some ordering of the vertices in $V(G)$ such that each vertex has at most $d$ neighbors among the preceding vertices.  The \emph{degeneracy of $G$} is the smallest $d$ such that $G$ is $d$-degenerate.  Notice that if graph $G$ is $d$-degenerate, then its $(d+1)$-core is the empty graph.  The coloring number of a graph $G$, denoted $\col(G)$, is equal to its degeneracy plus $1$.  It is well known that the coloring number of a graph is an upper bound on its DP-chromatic number.    

Relaxations of degeneracy have recently taken the literature by storm!  To review this, we define an operation that was presented in~\cite{BZ24} and studied by many others~\cite{BL23, BLS25, HHZ23, Y24, ZZZ23}.

\begin{define}[\cite{BZ24}]
    Suppose $G$ is a graph and let $f:V(G)\to\N$ be a function.  For a vertex $u\in V(G)$ and $W\subseteq N_G(u)$, the operation $\DSa(G,f,u,W)$ outputs the (possibly empty) graph $G_o$ and function $f_o:V(G_o)\to\{0,1,2,\ldots\}$ where
    \begin{equation*}
        G_o=G-u
        \hspace{0.5in}\text{and}\hspace{0.5in}
        f_o(v)=\left\{\begin{array}{l l}f(v)-1&\text{if $v\in N_{G}(u)- W$}\\
        f(v)&\text{otherwise.}\end{array}\right.
    \end{equation*}
\end{define}

The operation $\DSa(G,f,u,W)$ is \emph{legal} if $f(u)>\sum_{w\in W}f(w)$ and the resulting $f_o$ is positive for all $v\in V(G_o)$.  Intuitively, a graph $G$ is \emph{$ST^{(3)}$-$f$-degenerate} if there exists a sequence of legal $\DSa$ operations that takes $G$ to the empty graph.  Likewise, a graph $G$ is \emph{$ST^{(4)}$-$f$-degenerate} if there exists a sequence of legal $\DSa$ operations that takes $G$ to the empty graph while the size of the fourth input is at most $1$ for each operation.  For rigorous definitions see Section~\ref{sec:defAndFound}.  For $k\in\N$ and $j\in\{3,4\}$, we say that $G$ is \emph{$ST^{(j)}$-$k$-degenerate} if it is $ST^{(j)}$-$f$-degenerate when $f$ is identically $k$.  For $j\in\{3,4\}$, the \emph{strict type-$j$ degeneracy} of $G$ is denoted $ST^{(j)}(G)$ where
\begin{equation*}
	ST^{(j)}(G)=\min\left\{k:G\text{ is $ST^{(j)}$-$k$-degenerate}\right\}.
\end{equation*}
In $2023$, Han, He, and Zhu~\cite{HHZ23} introduced the notion of strict type-$3$ degeneracy (under a different name).  In $2021$, Bernshteyn and Lee~\cite{BL23} introduced the notion of weak degeneracy, denoted $wd(G)$ for graph $G$, which satisfies $ST^{(4)}(G)=wd(G)+1$ and $wd(G)+1\leq\col(G)$.   In $2024$, Zhou, Zhu, and Zhu~\cite{ZZZ23} showed that the strict type-$3$ degeneracy of a graph bounds the DP-chromatic number of the graph from above.  In the same paper they defined the strict type-$1$ and strict type-$2$ degeneracy of a graph. It is shown in~\cite{ZZZ23} that the strict type-$1$ degeneracy of a graph is an upper bound for the list chromatic number of the graph; however, it is unknown whether it is an upper bound on the DP-chromatic number of the graph.  It is also shown that the strict type-$2$ degeneracy of a graph is an upper bound for the DP-chromatic number of the graph.  In summary,

\begin{equation*}
	\chi(G)\leq\chi_{\ell}(G)\leq\chi_{_{DP}}(G)\leq ST^{(2)}(G) \leq ST^{(3)}(G)\leq ST^{(4)}(G)= wd(G)+1\leq \col(G).
\end{equation*}
It is also worth mentioning that in~\cite{ZZZ23} it was shown that $AT(G) \leq ST^{(1)}(G)$ where $AT(G)$ is the Alon-Tarsi number of graph $G$.  The Alon-Tarsi number of a graph is a well-studied upper bound on the list chromatic number of the graph, though it is not an upper bound on the DP-chromatic number of the graph~\cite{KM20}.  See~\cite{KM19} for the formal definition of the Alon-Tarsi number. 

The motivation for this paper is to extend the strict degeneracy notions in~\cite{ZZZ23} to the fractional context.  We have so far been unable to extend the notions of strict type-1 and strict type-2 degeneracy to the fractional context in any useful way.  In this paper we extend strict type-$3$ and strict type-$4$ degeneracy to the fractional context.

We now informally introduce our extensions.  The formal definitions are presented in Section~\ref{sec:defAndFound}.  For a graph $G$ and $j \in \{3,4\}$: the strict type-$j$ degeneracy of $G$ $ST^{(j)}(G)$ is extended to the fractional strict type-$j$ degeneracy of $G$ $ST^{(j)*}(G)$.  From the definitions it will be easy to see: $ST^{(3)*}(G)\leq ST^{(3)}(G)$, $ST^{(4)*}(G)\leq ST^{(4)}(G)$, and $ST^{(3)*}(G)\leq ST^{(4)*}(G).$

\subsection{Outline of results}\label{subsection:Outline}

In Section~\ref{sec:defAndFound} we provide rigorous definitions for fractional strict type-$3$ and fractional strict type-$4$ degeneracy.  We also prove a number of foundational results, some of which will be used throughout the remainder of the paper.  One particularly important result states the fractional strict type-$3$ degeneracy of a graph is an upper bound on the fractional DP chromatic number of the graph.

\begin{thm} \label{thm: upbound}
For any graph $G$, $\chi^*_{_{DP}}(G)\leq ST^{(3)*}(G)\leq ST^{(4)*}(G)$.
\end{thm}

In Section~\ref{sec:upperBounds} we prove theorems that give upper bounds on the fractional strict type-$4$ degeneracy of various classes of graphs.  We give special emphasis to classes of graphs where the fractional strict type-$4$ degeneracy of a graph gives us the best known upper bound on the fractional DP chromatic number of the graph.  We begin by showing the fractional strict type-$4$ degeneracy of a cycle equals its fractional DP chromatic number.

\begin{thm}\label{thm:c2r}
	For all $r\in\N$ satisfying $r\geq2$, $ST^{(4)*}(C_{2r})=2$.  Consequently $ST^{(4)*}(C_{2r})=ST^{(3)*}(C_{2r})=\chi_{_{DP}}^*(C_{2r})=\chi_{\ell}^*(C_{2r})=\chi^*(C_{2r})=2$.
\end{thm}

\begin{thm}\label{thm:c2r+1}
	For all $r\in\N$, $ST^{(4)*}(C_{2r+1})=2+1/r$.  Consequently $ST^{(4)*}(C_{2r+1})=ST^{(3)*}(C_{2r+1})=\chi_{_{DP}}^*(C_{2r+1})=\chi_{\ell}^*(C_{2r+1})=\chi^*(C_{2r+1})=2+1/r$.
\end{thm}

Then, we generalize these theorems by showing $ST^{(4)*}(G)=ST^{(3)*}(G)=\chi_{_{DP}}^*(G)$ when $G$ is a unicyclic graph, see Corollary~\ref{cor:unicyclic} below. Theorem~\ref{thm:c2r+1} gives an alternative (and quite different) proof of Theorem~\ref{thm: oddcycle} compared to the proof given in~\cite{DKM25}.  Next, we prove the following theorem which provides an improvement on the bound for the fractional DP-chromatic number given by Theorem~\ref{thm:completeMultipartite} for $K_{2,3}$, $K_{2,4}$, and $K_{2,5}$.  In fact, as far as we are aware it gives the best known upper bound on the fractional DP-chromatic number of these three graphs.

\begin{thm}\label{thm:Km,nUpperBound}
	For all $m,n\in\N$ satisfying $2\leq m\leq n$, $ST^{(4)*}(K_{m,n})\leq m+1-m/n$.  Consequently $\chi_{_{DP}}^*(K_{m,n})\leq m+1-m/n=\col(K_{m,n})-m/n$.
\end{thm}

Theorems~\ref{thm:c2r}, \ref{thm:c2r+1}, and~\ref{thm:Km,nUpperBound} provide an upper bound on the fractional DP-chromatic numbers of highly structured graphs.  We end Section~\ref{sec:upperBounds} by addressing graphs that are formed by subdividing the edges of an arbitrarily chosen graph.

\begin{thm}\label{thm:subdivisions}
    Suppose $H$ is a connected graph and $G$ is a subdivision of $H$ where all of the edges are replaced by internally disjoint paths each of whose lengths is at least $2r$ for some $r\geq2$.  Then $\chi^*_{_{DP}}(G)\leq ST^{(4)*}(G)\leq2+1/r$.
\end{thm}

Theorem~\ref{thm: fracDP2} says that for a connected graph $G$, $\chi_{_{DP}}^*(G)=2$ if and only if $G$ contains a single even cycle and no odd cycles.  Consequently, any connected graph containing more than one cycle has fractional DP-chromatic number strictly greater than $2$. Theorem~\ref{thm:subdivisions} gives us a way to construct connected graphs with many cycles whose fractional DP-chromatic numbers are close to $2$.

Corollary~\ref{cor:unicyclic} tells us connected graphs $G$ with average degree 2 satisfy $ST^{(4)*}(G)=\chi_{_{DP}}^*(G)$.  Connected graphs $T$ of average degree less than 2 satisfy this equality since $\chi^{*}(T) = \col(T)$.  The results of Section~\ref{sec:upperBounds} seem to indicate that strict type-4 fractional degeneracy is useful for bounding the fractional DP-chromatic numbers of graphs with average degree close to $2$. 

 In general it would be interesting to determine the graphs whose strict type-4 fractional degeneracy equals its fractional DP-chromatic number.

\begin{ques}\label{ques:slack}
    What graphs $G$ satisfy $ST^{(4)*}(G)=\chi^*_{_{DP}}(G)$?
\end{ques}
As previously stated, unicyclic graphs and trees are two classes of graphs that satisfy the equation in Question~\ref{ques:slack}.  It is also worth noting that chordal graphs satisfy this equation since it is well-known that $\chi^*(G) = \col(G)$ when $G$ is chordal.  In light of Question~\ref{ques:slack} we turn our attention to lower bounds on the fractional strict type-$4$ degeneracy of complete bipartite graphs in Section~\ref{sec:lowerBound}.  Specifically, we prove the following result.

\begin{thm}\label{thm:CompleteBipartiteLowerBound}
	For any $\epsilon\in(0,1]$, $ST^{(4)*}(K_{m,n})\geq m+1-\epsilon$ whenever $m\in\N$ and $n\geq 5m/\epsilon$.
\end{thm}

Theorem~\ref{thm:CompleteBipartiteLowerBound} shows the limitation in using strict type-$4$ degeneracy to bound the fractional DP-chromatic number of complete bipartite graphs. For example, as a consequence of Corollary $7$ from~\cite{DKM25} we know that $\chi_{_{DP}}^*(K_{m,n})\leq0.5m+1.75$ for all $m,n\in\N$ whereas Theorem~\ref{thm:CompleteBipartiteLowerBound} shows that $ST^{(4)*}(K_{m,n})$ can be arbitrarily close to $m+1$ provided $n$ is sufficiently large.  This shows that for any $r\in\R^+$ there exists a graph $G$ such that $ST^{(4)*}(G)-\chi^*_{_{DP}}(G)>r$.  In~\cite{BZ24} it is shown that for any $N \in \N$ there is a graph $G$ such that $ST^{(4)}(G) - ST^{(3)}(G) > N$.  Theorem~\ref{thm:CompleteBipartiteLowerBound} can be viewed as a potential first step towards extending this result to the fractional context.
\begin{ques}
   For each $r \in \R^+$ is there a graph $G$ such that $ST^{(4)*}(G)-ST^{(3)*}(G) > r$?
\end{ques}
Provocatively, we are not even aware of a graph $G$ satisfying $ST^{(4)*}(G)>ST^{(3)*}(G)$.

\section{Definitions and foundational results}\label{sec:defAndFound}

\subsection{Fractional strict type-$3$ and fractional strict type-$4$ Degeneracy}\label{subsec:wfdIntro}

In this section we give the formal definitions of fractional strict type-$3$ and fractional strict type-$4$ degeneracy.  First, we extend the $\DSa$ operation to the fractional context.

\begin{define}
	$\ShSa$ is an operation that takes as input: a graph $G$, two functions $S:V(G)\to\{0,1,2,\ldots\}$ and $T:V(G)\to\N$, a vertex $u \in V(G)$, and some $W \subseteq N_G(u)$.  In particular, an application of the $\ShSa$ operation with these inputs, written $\ShSa(G,S,T,u,W)$, outputs a graph $G_o$ ($G_o$ may be the empty graph) and an ordered pair of functions $(S_o,T_o)$ with $S_o:V(G_o)\to\{-1,0,1,2,\ldots\}$ and $T_o:V(G_o)\to\N$ according to the following.
	\begin{itemize}
		\item If $T(u)=1$, then
		\begin{align*}
			G_o&=G-u,\\
			S_o(v)&=\left\{\begin{array}{l l}S(v)-1&\text{if }uv\in E(G)\text{ and }v\not\in W\\S(v)&\text{otherwise},\end{array}\right.\\
			T_o(v)&=T(v).
		\end{align*}
		\item If $T(u) > 1$, then
		\begin{align*}
			G_o&=G,\\
			S_o(v)&=\left\{\begin{array}{l l}S(v)-1&\text{if }uv\in E(G)\text{ and }v\not\in W\\S(v)&\text{otherwise},\end{array}\right.\\
			T_o(v)&=\left\{\begin{array}{l l}T(v)-1&\text{if }v=u\\T(v)&\text{otherwise}.\end{array}\right.
		\end{align*}
	\end{itemize}
    In the special case when $W=\emptyset$, we will use the notation $\Sh(G,S,T,u)$ for \\ $\ShSa(G,S,T,u,\emptyset)$.

\end{define}

When we write $s=\ShSa(G,S,T,u,W)$, we adopt the convention that $s$ is not the output of $\ShSa(G,S,T,u,W)$; rather, we say that $s$ is the application of the $\ShSa$ operation to $G$, $S$, $T$, $u$, and $W$. Suppose $s=\ShSa(G,S,T,u,W)$. We say \emph{$s$ shades $u$}.  In the case where $T(u)=1$, we say \emph{$s$ removes $u$} or we say \emph{$u$ is removed from $G$ by $s$}.  Moreover, we say \emph{$s$ saves on $w$} for each $w\in W$.  Finally for any $\Omega \subseteq W$, we say \emph{$s$ saves on $\Omega$}.  Figure~\ref{fig:IntroSample} is a pictorial illustration of an application of the $\ShSa$ operation.

\begin{figure}[h!]

    \begin{center}\begin{tikzpicture}[scale=0.8,every node/.style={circle, draw, fill, inner sep=0pt, minimum size=4pt, text=black}]

			\node [anchor=center, label={[shift={(0.1,-0.2)}]left:$v_1\hspace{-1mm}:\hspace{-1mm}(3,1)$}, color=black] (0) at (2.25,-0.75) {};
			\node [anchor=center, label={[shift={(0,-0.5)}]above:$v_2\hspace{-1mm}:\hspace{-1mm}(2,2)$}, color=black] (1) at (3.5,-0.25) {};
			\node [anchor=center, label={[shift={(0,0.5)}]below:$v_3\hspace{-1mm}:\hspace{-1mm}(2,2)$}, color=black] (2) at (3.5,-1.25) {};
			\node [anchor=center, label={[shift={(-0.1,0.2)}]right:$v_4\hspace{-1mm}:\hspace{-1mm}(3,2)$}, color=black] (3) at (4.75,-0.75) {};

			\draw (0.center) to (1.center);
			\draw (0.center) to (2.center);
			\draw (3.center) to (1.center);
			\draw (3.center) to (2.center);

            \draw [->, thick] (5.9,-0.9) to (11.6,-0.9);
            \node [anchor=center, label={center:$\ShSa(G,S,T,v_4,\{v_3\})$}, color=white] (5) at (8.75,-1.25) {};

			\node [anchor=center, label={[shift={(0.1,0.2)}]left:$v_1\hspace{-1mm}:\hspace{-1mm}(3,1)$}, color=black] (0) at (12.75,-.75) {};
			\node [anchor=center, label={[shift={(0,-0.5)}]above:$v_2\hspace{-1mm}:\hspace{-1mm}(1,2)$}, color=black] (1) at (14,-.25) {};
			\node [anchor=center, label={[shift={(0,0.5)}]below:$v_3\hspace{-1mm}:\hspace{-1mm}(2,2)$}, color=black] (2) at (14,-1.25) {};
			\node [anchor=center, label={[shift={(-0.1,-0.2)}]right:$v_4\hspace{-1mm}:\hspace{-1mm}(3,1)$}, color=black] (3) at (15.25,-.75) {};

			\draw (0.center) to (1.center);
			\draw (0.center) to (2.center);
			\draw (3.center) to (1.center);
			\draw (3.center) to (2.center);

		\end{tikzpicture}\end{center}
        \vspace{-2.25cm}

    \caption{Pictorial illustration of an application of the $\ShSa$ operation with inputs $G=C_4$, $S$, $T$, $v_4$, and $\{v_3\}$.  On the left, for each vertex $v_i\in V(G)$, the ordered pair next to $v_i$ is $(S(v_i),T(v_i))$.  Suppose $G_o$ and $(S_o,T_o)$ is the output of $\ShSa(G,S,T,v_4,\{v_3\})$.  On the right, for each vertex $v_i\in V(G_o)$, the ordered pair next to $v_i$ is $(S_o(v_i),T_o(v_i))$.}
    \label{fig:IntroSample}
\end{figure}
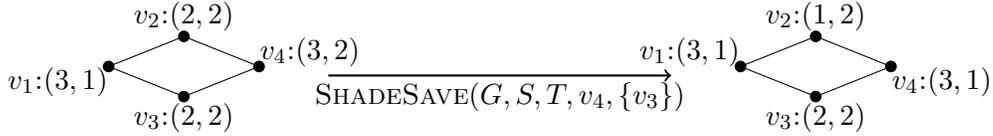

If $G_o$ and $(S_o,T_o)$ is the output of $\ShSa(G,S,T,u,W)$, we say $\ShSa(G,S,T,u,W)$ is \emph{legal} if $S(u)+T(u)>\sum_{w\in W}(S(w)+T(w))$ and $S_o$ is nonnegative for all $v\in V(G_o)$.  In this paper we will refer to any function with domain $V(G)$ and codomain $\{0,1,2,\ldots\}$ as a \emph{shield of $G$}, and we will refer to any function with domain $V(G)$ and codomain $\N$ as a \emph{target of $G$}.  Consequently, if $\ShSa(G,S,T,u,W)$ is legal then $S_o$ is a shield of $G_o$.

Suppose $k \in \N$, $G$ is a nonempty graph, $S$ is a shield of $G$, and $T$ is a target of $G$.  Suppose $\mathcal{S}=(s_1,\ldots,s_k)$ and $s_i=\ShSa(G_i,S_i,T_i,u_i,W_i)$ for each $i \in [k]$.  We say $\mathcal{S}$ is a \emph{$(G,S,T)$-legal sequence} if $G_1=G$, $S_1=S$, $T_1=T$, $G_{i+1}$ and $(S_{i+1},T_{i+1})$ are the outputs of $s_i$ for each $i \in [k-1]$, and $s_i$ is legal for all $i\in[k]$~\footnote{We will use this naming scheme for the inputs and outputs of the operations of a $(G,S,T)$-legal sequence unless otherwise noted.}. Furthermore, we say $\mathcal{S}$ is a \emph{restricted $(G,S,T)$-legal sequence} if $|W_i|\in\{0,1\}$ for all $i\in[k]$.  In the special case where $G$ is the empty graph, since $V(G)=\emptyset$, the empty function is the only function that is a shield or target of $G$.  For any graph $G$, including the empty graph, we take the empty sequence to be a restricted $(G,S,T)$-legal sequence.  When $\mathcal{S}$ is a nonempty $(G,S,T)$-legal sequence, the \emph{final outputs} of  $\mathcal{S}$ are the outputs of the last term of $\mathcal{S}$.  When $\mathcal{S}$ is an empty $(G,S,T)$-legal sequence, the \emph{final outputs} of  $\mathcal{S}$ are $G$ and $(S,T)$. We write $G_\mathcal{S}$ and $(S_{\mathcal{S}},T_{\mathcal{S}})$ for the final outputs of a $(G,S,T)$-legal sequence $\mathcal{S}$.  In the special case where we only consider a single $(G,S,T)$-legal sequence with $q$ terms, we will write $G_{q+1}$ and $(S_{q+1},T_{q+1})$ for its final outputs.  When constructing a $(G,S,T)$-legal sequence $(s_1, \ldots, s_k)$, if we wish to define $s_i$ as $\ShSa(G_i,S_i,T_i,u_i,W_i)$  we will often write $s_i = \ShSa(u_i,W_i)$ for short when the first three inputs are clear from context.  Similarly, if we wish to define $s_i$ as $\Sh(G_i,S_i,T_i,u_i)$  we will often write $s_i = \Sh(u_i)$ for short.  See Figure~\ref{fig:IntroSample2} for a pictorial illustration of a $(G,S,T)$-legal sequence.

\begin{figure}[h!]

    \begin{center}\begin{tikzpicture}[scale=0.8,every node/.style={circle, draw, fill, inner sep=0pt, minimum size=4pt, text=black}]
			\node [anchor=center, label={[shift={(0.1,-0.2)}]left:$v_1\hspace{-1mm}:\hspace{-1mm}(3,2)$}, color=black] (0) at (-1.25,0) {};
			\node [anchor=center, label={[shift={(0.3,-0.5)}]above:$v_2\hspace{-1mm}:\hspace{-1mm}(3,2)$}, color=black] (1) at (0,0.5) {};
			\node [anchor=center, label={[shift={(-0.3,0.5)}]below:$v_3\hspace{-1mm}:\hspace{-1mm}(3,2)$}, color=black] (2) at (0,-0.5) {};
			\node [anchor=center, label={[shift={(-0.1,0.2)}]right:$v_4\hspace{-1mm}:\hspace{-1mm}(3,2)$}, color=black] (3) at (1.25,0) {};

			\draw (0.center) to (1.center);
			\draw (0.center) to (2.center);
			\draw (3.center) to (1.center);
			\draw (3.center) to (2.center);

            \draw [->, thick] (0,-1.1) to (2.75,-1.75);
            \node [anchor=center, label={center:$s_1=\Sh(v_1)$}, color=white] (5) at (0.1,-1.75) {};

			\node [anchor=center, label={[shift={(0.1,-0.2)}]left:$v_1\hspace{-1mm}:\hspace{-1mm}(3,1)$}, color=black] (0) at (2.75,-0.75) {};
			\node [anchor=center, label={[shift={(0.3,-0.5)}]above:$v_2\hspace{-1mm}:\hspace{-1mm}(2,2)$}, color=black] (1) at (4,-0.25) {};
			\node [anchor=center, label={[shift={(-0.3,0.5)}]below:$v_3\hspace{-1mm}:\hspace{-1mm}(2,2)$}, color=black] (2) at (4,-1.25) {};
			\node [anchor=center, label={[shift={(-0.1,0.2)}]right:$v_4\hspace{-1mm}:\hspace{-1mm}(3,2)$}, color=black] (3) at (5.25,-0.75) {};

			\draw (0.center) to (1.center);
			\draw (0.center) to (2.center);
			\draw (3.center) to (1.center);
			\draw (3.center) to (2.center);

            \draw [->, thick] (4,-1.85) to (6.75,-2.5);
            \node [anchor=center, label={center:$s_2=\ShSa(v_4,\{v_3\})$}, color=white] (5) at (3.1,-2.5) {};

			\node [anchor=center, label={[shift={(0.1,-0.2)}]left:$v_1\hspace{-1mm}:\hspace{-1mm}(3,1)$}, color=black] (0) at (6.75,-1.5) {};
			\node [anchor=center, label={[shift={(0.3,-0.5)}]above:$v_2\hspace{-1mm}:\hspace{-1mm}(1,2)$}, color=black] (1) at (8,-1) {};
			\node [anchor=center, label={[shift={(-0.3,0.5)}]below:$v_3\hspace{-1mm}:\hspace{-1mm}(2,2)$}, color=black] (2) at (8,-2) {};
			\node [anchor=center, label={[shift={(-0.1,0.2)}]right:$v_4\hspace{-1mm}:\hspace{-1mm}(3,1)$}, color=black] (3) at (9.25,-1.5) {};

			\draw (0.center) to (1.center);
			\draw (0.center) to (2.center);
			\draw (3.center) to (1.center);
			\draw (3.center) to (2.center);

            \draw [->, thick] (8,-2.85) to (10.75,-3.5);
            \node [anchor=center, label={center:$s_3=\ShSa(v_4,\{v_2\})$}, color=white] (5) at (7.1,-3.5) {};

			\node [anchor=center, label={[shift={(0.1,-0.2)}]left:$v_1\hspace{-1mm}:\hspace{-1mm}(3,1)$}, color=black] (0) at (10.75,-2.5) {};
			\node [anchor=center, label={[shift={(0.3,-0.5)}]above:$v_2\hspace{-1mm}:\hspace{-1mm}(1,2)$}, color=black] (1) at (12,-2) {};
			\node [anchor=center, label={[shift={(-0.3,0.5)}]below:$v_3\hspace{-1mm}:\hspace{-1mm}(1,2)$}, color=black] (2) at (12,-3) {};
			\node [anchor=center, label={[shift={(-0.1,0.2)}]right:$v_4$}, color=black, fill=white] (3) at (13.25,-2.5) {};

			\draw (0.center) to (1.center);
			\draw (0.center) to (2.center);

		\end{tikzpicture}\end{center}
        \vspace{-1.75cm}

    \caption{
    Pictorial illustration of the $(G,S,T)$-legal sequence $\mathcal{S}=(s_1,s_2,s_3)$ where $G=C_4$, $S(v)=3$, and $T(v)=2$.  Suppose $j \in [3]$.  For the diagram to the left of the arrow corresponding to $s_j$, the ordered pair next to each vertex is $(S_j(v_i),T_j(v_i))$ for each $i\in[4]$.  The right-most diagram is a pictorial illustration of the final output of $\mathcal{S}$ (note that $v_4$ has been removed from $G_3$ by $s_3$).}
    \label{fig:IntroSample2}
\end{figure}
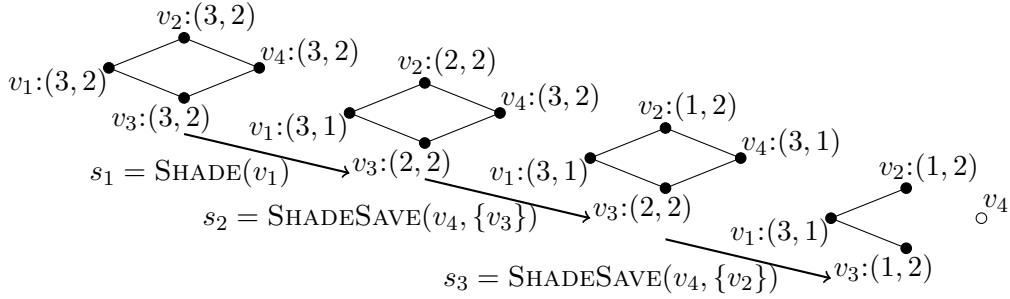

We say a $(G,S,T)$-legal sequence $\mathcal{S}$ is a \emph{complete $(G,S,T)$-legal sequence} if $G_{\mathcal{S}}$ is the empty graph.  A $(G,S,T)$-legal sequence $\mathcal{S}$ is a \emph{complete, restricted $(G,S,T)$-legal sequence} when $\mathcal{S}$ is a complete $(G,S,T)$-legal sequence and a restricted $(G,S,T)$-legal sequence.   Note that in the case where $G$ is the empty graph, the empty sequence is a complete, restricted $(G,S,T)$-legal sequence.  Notice that the length of a complete $(G,S,T)$-legal sequence is $\sum_{v\in V(G)}T(v)$.  

Graph $G$ is \emph{$ST^{(3)}$-$(S,T)$-degenerate} if there exists a complete $(G,S,T)$-legal sequence.  Likewise, $G$ is \emph{$ST^{(4)}$-$(S,T)$-degenerate} if there exists a complete, restricted $(G,S,T)$-legal sequence.  Suppose $f: V(G) \rightarrow \N$ is a function.  We now define what it means for a graph to be $ST^{(j)}$-$f$-degenerate for $j\in\{3,4\}$ which were notions informally discussed in Section~\ref{intro}.  Let $g(v)=f(v)-1$ and $h(v)=1$ for all $v\in V(G)$. Then $G$ is \emph{$ST^{(j)}$-$f$-degenerate} if $G$ is $ST^{(j)}$-$(g,h)$-degenerate.  These definitions are equivalent to those given in~\cite{ZZZ23}.

 For $s\in\{0,1,2,\ldots\}$, $t\in\N$, and $j\in\{3,4\}$, we say that $G$ is \emph{$ST^{(j)}$-$(s,t)$-degenerate} if it is $ST^{(j)}$-$(S,T)$-degenerate where $S$ is the shield of $G$ that is identically $s$ and $T$ is the target of $G$ that is identically $t$ (in this case we will refer to a complete $(G,S,T)$-legal sequence as a \emph{complete $(G,s,t)$-legal sequence}).
 
 For $j\in\{3,4\}$, the \emph{fractional strict type-$j$ degeneracy} of $G$, denoted $ST^{(j)*}(G)$, is given by
\begin{equation*}
	ST^{(j)*}(G)=\inf\left\{\frac{s+t}{t}:G\text{ is $ST^{(j)}$-$(s,t)$-degenerate}\right\}.
\end{equation*}
Clearly, for any graph $G$ $ST^{(3)*}(G)\leq ST^{(4)*}(G)$. For $j\in\{3,4\}$, a graph $G$ is $ST^{(j)}$-$(s,1)$-degenerate if and only if $G$ is $ST^{(j)}$-$(s+1)$-degenerate; consequently, $ST^{(j)*}(G) \leq ST^{(j)}(G)$.

For edgeless graphs the fractional strict type-3 and fractional strict type-4 degeneracy are easy to compute since we can construct a complete, restricted, legal sequence by repeatedly shading each vertex of such a graph until we achieve the empty graph (regardless of the specific shield and target of $G$).
\begin{obs}\label{obs:Trivial}
Suppose $G$ is a (possibly empty) graph where $|E(G)|=0$, $S$ is a shield of $G$, $T$ is a target of $G$, and $j\in\{3,4\}$.  Then, $G$ is $ST^{(j)}$-$(S,T)$-degenerate. Consequently, $ST^{(j)*}(G) = 1$.  
\end{obs}

In some graph $G$, if for all $v\in V(G)$ the values of $S(v)$ are large enough in relation to the values of $T(w)$ for all $w\in N_G(v)$, it is easy to construct a complete, restricted $(G,S,T)$-legal sequence by greedily shading the vertices.  The next observation makes this precise.

\begin{obs}\label{obs:DegenerateByBackDegree}
    Suppose $v_1,\ldots,v_n$ is an arbitrary ordering of the vertices of graph $G$.  Let $N_i^-=\{v_j:v_j\in N_G(v_i)\text{ and }j\in[i-1]\}$. If $S$ and $T$ are a shield and a target of $G$ satisfying $S(v_i)\geq\sum_{v_j\in N_i^-}T(v_j)$ for each $i\in[n]$ then $G$ is $ST^{(4)}$-$(S,T)$-degenerate.  Consequently $G$ is $ST^{(4)}$-$(\col(G)-1,1)$-degenerate and $ST^{(4)*}(G)\leq\col(G)$.
\end{obs}

The complete, restricted $(G,S,T)$-legal sequence that one would construct to prove Observation~\ref{obs:DegenerateByBackDegree} would only need to consist of shade operations.  As we will see below, when fractional strict type-$4$ degeneracy gives the best known bound on the fractional DP-chromatic number of a graph, we will often construct complete, restricted, legal sequences that include operations that save on vertices.  Finally, to ensure thoroughness and clarity, we make a technical remark.
\begin{rem}\label{rem:inductiveDefinition}
    Suppose $n$ is an integer such that $n \geq 2$, $G$ is a graph, $S$ is a shield of $G$, and $T$ is a target of $G$.  In many of our proofs we wish to inductively construct an $n$-term $(G,S,T)$-legal sequence with $i^{th}$ term denoted $s_i$.  To do this we define $s_1$, then for each $i\in[2:n]$ we define $s_i$ based on the assumption that $s_1,\ldots,s_{i-1}$ are already defined and legal.   We adopt the convention that if $s_i$ is defined so that it is not legal, our construction terminates.  Indeed, when $i\in [2:n]$ and $s_i$ is defined so that it is not legal, our construction only produces the $(G,S,T)$-legal sequence $(s_1,\ldots,s_{i-1})$.  Note that in the case where $s_1$ is not legal, then our construction produces the empty sequence, which is a $(G,S,T)$-legal sequence by definition. 
\end{rem}

\subsection{Foundational results}\label{subsec:foundResults}
 We begin with two observations that are easy to prove.
 
\begin{obs}\label{obs:Sequence}
For $k \in \N$ and $j\in\{3,4\}$ suppose graph $G$ is $ST^{(j)}$-$(S,T)$-degenerate.  Suppose $(s_1, \ldots, s_k)$ is a complete $(G,S,T)$-legal sequence (in the case where $j=4$ suppose this sequence is restricted as well).  Recall for $q\in[2:k+1]$ that $G_q$ (which may be the empty graph) and $(S_q,T_q)$ are the outputs of $s_{q-1}$.  Then $\sum_{v\in V(G)}T(v)-\sum_{v\in V(G_q)}T_q(v) = q-1$ and $G_q$ is $ST^{(j)}$-$(S_q,T_q)$-degenerate.
\end{obs}

\begin{obs}\label{obs:Components}
    Suppose $G$ is a graph with components $H_1,\ldots,H_n$ and $j\in\{3,4\}$.  Also, suppose $H_i$ is $ST^{(j)}$-$(S_i,T_i)$-degenerate for each $i\in[n]$.  If $S$ is the shield of $G$ and $T$ is the target of $G$ such that for each $i\in[n]$ and $v\in V(H_i)$, $S(v)=S_i(v)$ and $T(v)=T_i(v)$, then $G$ is $ST^{(j)}$-$(S,T)$-degenerate.
\end{obs}

In many of our proofs, we will show that a complete $(G,S,T)$-legal sequences exists by constructing a shorter $(G,S,T)$-legal sequence, and then showing that it is the beginning of a complete $(G,S,T)$-legal sequence.  We now introduce some terminology that will be useful for such constructions.

Suppose $\mathcal{S}_1 = (s_1, \ldots, s_k)$ and $\mathcal{S}_2 = (x_1, \ldots, x_t)$ are two nonempty $(G,S,T)$-legal sequences with $k \leq t$.  We say that $\mathcal{S}_2$ is an \emph{extension} of $\mathcal{S}_1$ if $s_i = x_i$ for each $i \in [k]$; furthermore, if $\mathcal{S}_1$ and $\mathcal{S}_2$ are also both restricted $(G,S,T)$-legal sequences, then we say that $\mathcal{S}_2$ is a \emph{restricted extension} of $\mathcal{S}_1$.  We say that $\mathcal{S}_2$ is a \emph{complete extension} of $\mathcal{S}_1$ if $\mathcal{S}_2$ is an extension of $\mathcal{S}_1$ and $\mathcal{S}_2$ is a complete $(G,S,T)$-legal sequence; furthermore, if $\mathcal{S}_1$ and $\mathcal{S}_2$ are also both restricted $(G,S,T)$-legal sequences, then we say that $\mathcal{S}_2$ is a \emph{complete, restricted extension} of $\mathcal{S}_1$.  

The following observation is easily obtained by concatenating a sequence that we know exists to the end of a given $(G,S,T)$-legal sequence.

\begin{obs}\label{obs:ExtendingToComplete}
    If $\mathcal{S}$ is a nonempty $(G,S,T)$-legal sequence and $G_{\mathcal{S}}$ is $ST^{(3)}$-$(S_{\mathcal{S}},T_{\mathcal{S}})$-degenerate, then there is a complete extension of $\mathcal{S}$.

    If $\mathcal{S}$ is a nonempty, restricted $(G,S,T)$-legal sequence and $G_{\mathcal{S}}$ is $ST^{(4)}$-$(S_{\mathcal{S}},T_{\mathcal{S}})$-degenerate, then there is a complete, restricted extension of $\mathcal{S}$.
\end{obs}

For many of our results, we wish to prove some graph $G$ is $ST^{(4)}$-$(s,t)$-degenerate. However, some of our proofs require the flexibility offered by non-constant shield or target functions.  Consequently, we will frequently use the following lemma.

\begin{lem}\label{lem:Monotonicity}
 Let $G$ be a graph (possibly empty).  Suppose $S$ and $S'$ are shields of $G$ satisfying $S(v) \leq S'(v)$ for each $v \in V(G)$, and suppose $T$ and $T'$ are targets of $G$ satisfying $T(v) \geq T'(v)$ for each $v \in V(G)$.  Then, for each $j\in\{3,4\}$, if $G$ is $ST^{(j)}$-$(S,T)$-degenerate, then $G$ is $ST^{(j)}$-$(S',T')$-degenerate.
 \end{lem}

\begin{proof}
  We prove this when $j=3$.  Our proof can easily be modified so that it applies when $j=4$.  For the sake of contradiction, suppose the lemma is false.  Let $\mathcal{C}$ be the set of tuples $(G,S,T,S',T')$ that satisfy the hypotheses of the lemma but don't satisfy the conclusion.  For each $(G,S,T,S',T') \in \mathcal{C}$, let $\lambda(G,S,T,S',T') = \sum_{v \in V(G)} T(v)$.

  Now, among all elements of $\mathcal{C}$ suppose $(G,S,T,S',T')$ is an element so that $\lambda(G,S,T,S',T')$ is as small as possible.  Let $m = \lambda(G,S,T,S',T')$.  Clearly, $m > 0$ since the lemma holds when $G$ is the empty graph.  Suppose $\mathcal{S} = (s_1, \dots, s_m)$ is a complete $(G,S,T)$-legal sequence.  By Observation~\ref{obs:Sequence}, $G_2$ is $ST^{(3)}$-$(S_2,T_2)$-degenerate.  Suppose $s_1=\ShSa(v_1,W_1)$.  Clearly, $T'(v_1)< T(v_1)$ or $T'(v_1)=T(v_1)$.  We will obtain a contradiction in each of these cases.

  If $T'(v_1) < T(v_1)$, then $G_2=G$.  Clearly, $S_2(v) \leq S(v) \leq S'(v)$ for all $v \in V(G)$. Note that $T_2(v_1)=T(v_1)-1\geq T'(v_1)$ and $T_2(v)=T(v)\geq T'(v)$ for all $v\in V(G)\setminus\{v_1\}$.  So $T'(v)\leq T_2(v)$ for all $v\in V(G)$.  Since $(G,S,T,S',T')\in\mathcal{C}$, we know that $G$ is not $ST^{(3)}$-$(S',T')$-degenerate.  Consequently, $(G,S_2,T_2,S',T')\in\mathcal{C}$ but $\lambda(G,S_2,T_2,S',T')=m-1$ which is a contradiction to the minimality of $m$.

  If $T'(v_1)=T(v_1)$, then let $s'_1=\ShSa(G,S',T',v_1,W'_1)$, where $W'_1=\{w:w\in W_{1}\text{ and }S'(w)= S(w)\}$.  Since $s_1$ is legal we know 
     \begin{equation*}
        S'(v_{1})+T'(v_{1}) \geq S(v_{1})+T(v_{1}) > \sum_{v\in W_{1}}(S(v)+T(v)) \geq \sum_{v\in W'_{1}}(S'(v)+T'(v)).
     \end{equation*}
    Consequently, $s'_1$ is legal.  Let $G'_2$ and $(S'_2,T'_2)$ be the outputs of $s'_1$.  Since $T(v_1)=T'(v_1)$, we know $G_2=G'_2$.  It is also clear that $T'_2(v) \leq T_2(v)$ for each $v\in V(G_2)$.  For each $w\in W'_1$ we see that $S_2(w) = S(w) \leq S'(w) = S_2'(w)$.  For each $w\in W_1\setminus W'_1$ we have $S_2(w) = S(w) < S'(w)$ which implies $S'_2(w) = S'(w)-1 \geq S_2(w)$. For each $w\in V(G_2)\setminus W_1$ we have $S_2(w)-S(w)=S'_2(w)-S'(w)$ and $S'(w)\geq S(w)$ which implies $S'_2(w)\geq S_2(w)$.

    By Observation~\ref{obs:ExtendingToComplete}, since $G$ is not $ST^{(3)}$-$(S',T')$-degenerate, $G_2$ is not $ST^{(3)}$-$(S'_2,T'_2)$-degenerate.  Consequently $(G_2,S_2,T_2,S'_2,T'_2)\in\mathcal{C}$ but $\lambda(G_2,S_2,T_2,S'_2,T'_2)=m-1$ which is a contradiction to the minimality of $m$.
 \end{proof}

The following corollary is now immediate and will be used several times in the remainder of the paper.

\begin{cor}\label{cor:FractionalRemovabilityFromUnbalancedST}
    For each $j\in\{3,4\}$, if a nonempty graph $G$ is $ST^{(j)}$-$(S,T)$-degenerate, $s=\max_{v\in V(G)}S(v)$, and $t=\min_{v\in V(G)}T(v)$, then $G$ is $ST^{(j)}$-$(s,t)$-degenerate and $ST^{(j)*}(G)\leq(s+t)/t$.
\end{cor}

It is easy to show that if a graph is $(a,b)$-colorable, then it must be $(ka,kb)$-colorable for all $k\in\N$.  With this in mind,  Erd\H{o}s et al.~\cite{ERT} asked: If $G$ is $(a,b)$-choosable, does it follow that $G$ is $(ka,kb)$-choosable for each $k \in \N$?  Tuza and Voigt~\cite{TV96} showed that the answer to this question is yes when $a=2$ and $b=1$.  However, in~\cite{DH18}, a graph that is $4$-choosable but not $(8,2)$-choosable is constructed.  The analogue of this question in the context of fractional strict type-$3$ and fractional strict type-$4$ degeneracy has an answer of yes.

\begin{thm}\label{thm:ks,kt-degenerate}
    For each $j\in\{3,4\}$, $s\in\{0,1,2,\ldots\}$, and $t\in\N$, if graph $G$ (possibly empty) is $ST^{(j)}$-$(s,t)$-degenerate, then it is $ST^{(j)}$-$(ks,kt)$-degenerate for each $k\in\N$. 
\end{thm}

This theorem comes as an immediate consequence of the following lemma.

\begin{lem}\label{thm:kS,kT-degenerate}
    Suppose $j\in\{3,4\}$, $G$ is a (possibly empty) graph, and $S$ and $T$ are a shield and a target of $G$.  if $G$ is $ST^{(j)}$-$(S,T)$-degenerate, then it is $ST^{(j)}$-$(kS,kT)$-degenerate for each $k\in\N$. 
\end{lem}
\begin{proof}
     We prove this when $j=3$.  Our proof can easily be modified so that it applies when $j=4$.  For the sake of contradiction, suppose the lemma is false.  Let $\mathcal{C}$ be the set of all tuples $(G,S,T,k)$ that satisfy the hypothesis of the lemma but don't satisfy the conclusion.  For each $(G,S,T,k)\in\mathcal{C}$ let $\lambda(G,S,T,k)=\sum_{v\in V(G)}T(v)$.

    Now, among all elements of $\mathcal{C}$ suppose $(G,S,T,k)$ is an element so that $\lambda(G,S,T,k)$ is as small as possible.  Let $m=\lambda(G,S,T,k)$.  Clearly, $m>0$ since the empty graph is $ST^{(3)}$-$(kS,kT)$-degenerate.  Suppose $\mathcal{S}=(s_1,\ldots,s_m)$ is a complete $(G,S,T)$-legal sequence and $s_1=\ShSa(v_1,W_1)$.  Clearly $S(v_1)+T(v_1)\geq\sum_{w\in W_1}(S(w)+T(w))+1$ which means $kS(v_1)+kT(v_1)\geq\sum_{w\in W_1}(kS(w)+kT(w))+k$.

    For $G'=G$, $S'=kS$, and $T'=kT$, we will inductively construct a $(G',S',T')$-legal sequence of at most $k$ terms.  Suppose $\mathcal{S}'$ is our inductively constructed sequence with at most $k$ term where the $i^{th}$ term is $s'_i=\ShSa(v_1,W_1)$.  Now, the following statements can easily be proven by induction on $i$ for each $i \in [0:k]$ 
    
     \begin{enumerate}[(1)]
        \item $G'_{i+1} = G$ and $G'_{k+1} = G_2$,
        
        \item $\displaystyle
        S'_{i+1}(v)=\left\{\begin{array}{l l}
            kS(v)-i&\text{if }v\in N_G(v_1)\setminus W_1\\
            kS(v)&\text{otherwise}
        \end{array}\right.
        \hspace{0.05in}\text{and}\hspace{0.075in}
        T'_{i+1}(v)=\left\{\begin{array}{l l}
            kT(v)-i&\text{if }v=v_1\\
            kT(v)&\text{otherwise.}
        \end{array}\right.$

        Note in the case where $i=k$ that $T'_{k+1}$ is defined over the vertices in $V(G_2)$ and it is possible that $v_1\not\in V(G_2)$.
        
        \item  if $i \geq 1$, then $(s'_1, \dots, s'_i)$ is a $(G',S',T')$-legal sequence.  
    \end{enumerate}
    
    Now, 
    \begin{equation*}
        S'_{k+1}(v)=\left\{\begin{array}{l l}
            kS(v)-k&\text{if }v\in N_G(v_1)\setminus W_1\\
            kS(v)&\text{otherwise}
        \end{array}\right.
        \hspace{0.25in}\text{and}\hspace{0.25in}
        T'_{k+1}(v)=\left\{\begin{array}{l l}
            kT(v)-k&\text{if }v=v_1\\
            kT(v)&\text{otherwise.}
        \end{array}\right.
    \end{equation*}
    As a consequence of this, we see that $G'_{k+1}=G_2$, $S'_{k+1}=kS_2$, and $T'
    _{k+1}=kT_2$.  By Observation~\ref{obs:ExtendingToComplete} we see that $G_2$ is $ST^{(3)}$-$(S_2,T_2)$-degenerate but not $ST^{(3)}$-$(kS_2,kT_2)$-degenerate.  Consequently, $(G_2,S_2,T_2,k)\in\mathcal{C}$ but $\lambda(G_2,S_2,T_2,k)=m-1$ which is a contradiction to the minimality of $m$.
\end{proof}

Our initial interest in studying fractional strict type-$3$ and fractional strict type-$4$ degeneracy came from its application to fractional DP-coloring.  We will now prove that the fractional strict type-$3$ degeneracy of a graph is an upper bound on the fractional DP-chromatic number of the graph.  First we need a definition from~\cite{D25}.  For graph $G$ and functions $F_1:V(G)\to\N$ and $F_2:V(G)\to\N$, we say \emph{$G$ is $(F_1,F_2)$-DP-colorable} if for any cover $\mathcal{H}=(L,H)$ of $G$ such that $|L(v)|=F_1(v)$ for all $v\in V(G)$, there is an independent set of vertices $T\subseteq V(H)$ satisfying $|T\cap L(v)|=F_2(v)$ for all $v\in V(G)$.

\begin{lem}\label{lem:STDP}
Suppose $S$ and $T$ are a shield and target of graph $G$.  If $G$ is $ST^{(3)}$-$(S,T)$-degenerate, then $G$ is $(S+T,T)$-DP-colorable.  Consequently if $G$ is $ST^{(3)}$-$(s,t)$-degenerate then $\chi_{_{DP}}^*(G)\leq(s+t)/t$
\end{lem}

\begin{proof}
    For the sake of contradiction, suppose the lemma is false.  Let $\mathcal{C}$ be the set of all tuples $(G,S,T)$ that satisfy the hypothesis of the lemma but don't satisfy the conclusion.  For each $(G,S,T)\in\mathcal{C}$ let $\lambda(G,S,T)=\sum_{v\in V(G)}T(v)$.

    Now, among the elements of $\mathcal{C}$ suppose $(G,S,T)$ is an element so that $\lambda(G,S,T)$ is as small as possible.  Let $m=\lambda(G,S,T)$.  Clearly, $m>1$ since a copy of $K_1$ is $ST^{(3)}$-$(S,T)$-degenerate for any $S$ and $T$.

    Suppose $\mathcal{S}=(s_1,\ldots,s_m)$ is a complete $(G,S,T)$-legal sequence.  By Observation~\ref{obs:Sequence} we know that $G_2$ is $ST^{(3)}$-$(S_2,T_2)$-degenerate.  Also by Observation~\ref{obs:Sequence} we know that $\lambda(G_2,S_2,T_2)=m-1$; so, we know $(G_2,S_2,T_2)\not\in\mathcal{C}$ and $G_2$ is $(S_2+T_2,T_2)$-DP-colorable.

    Since $(G,S,T)\in\mathcal{C}$ there is a cover of $G$, $\mathcal{H}=(L,H)$, satisfying $|L(v)|=S(v)+T(v)$ for all $v\in V(G)$ such that there is no independent set of vertices $J \subseteq V(H)$ satisfying $|J\cap L(v)|=T(v)$ for all $v\in V(G)$.  Since $s_1=\ShSa(v_1,W_1)$ is legal, we know $S(v_1)+T(v_1)>\sum_{w\in W_1}(S(w)+T(w))$ which means $|L(v_1)|>|\bigcup_{w\in W_1}L(w)|$.  So, there is a $u \in L(v_1)$ such that $u$ has no neighbors (in $H$) in $\bigcup_{w\in W_1}L(w)$.  For each $v\in(N_G(v_1)\setminus W_1)$, let $a_v$ be the neighbor of $u$ in $L(v)$ or an arbitrary element of $L(v)$ if $u$ has no neighbor in $L(v)$.  Let $A=\{a_v : v\in N_G(v_1)\setminus W_1\}$.  Let $L_2:V(G_2)\to 2^{V(H)}$ be the function given by
    \begin{equation*}
        L_2(v)=\left\{\begin{array}{l l}
            L(v)\setminus\{u\}&\text{if }v=v_1\\
            L(v)\setminus\{a_v\}&\text{if }v\in(N_G(v_1)\setminus W_1)\\
            L(v)&\text{otherwise}.
        \end{array}\right.
    \end{equation*}
    Note that $L_2$ is defined over the vertices in $V(G_2)$ and it is possible that $v_1\not\in V(G_2)$.  Next, we let $H_2=H[\bigcup_{v\in V(G_2)}L_2(v)]$ and $\mathcal{H}_2=(L_2,H_2)$.  Clearly $\mathcal{H}_2$ is a cover of $G_2$ such that $|L_2(v)|=S_2(v)+T_2(v)$ for all $v\in V(G_2)$.  Since $G_2$ is $(S_2+T_2,T_2)$-DP-colorable, we know there exists an independent set of vertices $I_2\subseteq V(H_2)$ such that $|I_2\cap L(v)|=T_2(v)$ for all $v\in V(G_2)$.  Since all of the neighbors of $u$ in $H$ have been removed when constructing $H_2$, we know that $I=I_2\cup\{u\}$ is an independent set of vertices in $H$ satisfying $|I\cap L(v)|=T(v)$ for all $v\in V(G)$, a contradiction.
\end{proof}

The subsequent theorem follows immediately.

\begin{customthm} {\bf\ref{thm: upbound}}
For any graph $G$, $\chi^*_{_{DP}}(G)\leq ST^{(3)*}(G)\leq ST^{(4)*}(G)$.
\end{customthm}
\begin{proof}
    The second inequality follows immediately from the definitions.   Lemma~\ref{lem:STDP} says that for each $s\in\{0,1,2,\ldots\}$ and $t\in\N$ if $G$ is $ST^{(3)}$-$(s,t)$-degenerate, then $G$ is $(s+t,t)$-DP-colorable.  Consequently, $\{(s+t)/t:G\text{ is $ST^{(3)}$-$(s,t)$-degenerate}\} \subseteq \{a/b:G\text{ is $(a,b)$-DP-colorable}\}$.  Thus
    \begin{align*}
        \chi^*_{_{DP}}(G)&=\inf\{a/b:G\text{ is $(a,b)$-DP-colorable}\}\\
        &\leq \inf\{(s+t)/t:G\text{ is $ST^{(3)}$-$(s,t)$-degenerate}\} = ST^{(3)*}(G).
    \end{align*}
\end{proof}

Next, we present a lemma, which is a useful tool, that will be used when we generalize Theorems~\ref{thm:c2r} and~\ref{thm:c2r+1} to unicyclic graphs (Corollary~\ref{cor:unicyclic} below).

\begin{lem}\label{lem:appendingVertex}
    Suppose $G$ is a nonempty graph and $S$ and $T$ are a shield and target of $G$.  Also, suppose $G'$ is a graph such that $V(G')=V(G)\cup \{z\}$ for some $z \not\in V(G)$ and $G'[V(G)]=G$.  Finally, suppose $S'$ and $T'$ are a shield and target of $G'$ satisfying $S'(v)=S(v)$ and $T'(v)=T(v)$ for all $v\in V(G)$ while $S'(z)\geq\sum_{v\in N_{G'}(z)}T(v)$.  The following two statements hold.
    \begin{itemize}
        \item If $G$ is $ST^{(3)}$-$(S,T)$-degenerate then $G'$ is $ST^{(3)}$-$(S',T')$-degenerate.
        \item If $G$ is $ST^{(4)}$-$(S,T)$-degenerate then $G'$ is $ST^{(4)}$-$(S',T')$-degenerate.
    \end{itemize}
\end{lem}
Note that in the statement of Lemma~\ref{lem:appendingVertex}, aside from positivity which comes from the definition of target, no additional conditions are required for $T'(z)$.
\begin{proof}
    If $G$ is $ST^{(3)}$-$(S,T)$-degenerate, then we know there exists a complete $(G,S,T)$-legal sequence $\mathcal{S}=(s_1,\ldots,s_k)$ (in the case where $G$ is also $ST^{(4)}$-$(S,T)$-degenerate suppose that we select $\mathcal{S}$ to be a complete, restricted $(G,S,T)$-legal sequence).  We wish to inductively construct a $k$-term $(G',S',T')$-legal sequence.  Suppose $\mathcal{S}'$ is our inductively constructed sequence with at most $k$ terms where the $i^{th}$ term is $s'_i=\ShSa(v_i,W_i)$ (where $v_i$ and $W_i$ are the fourth and fifth inputs of $s_i$).

    Let $f: [k] \rightarrow (\N \cup \{0\})$ be the function that maps each $i \in [k]$ to the number of terms among $s_1, \ldots, s_i$ that have an element of $N_{G'}(z)$ as its fourth input.  Notice $f(i) \leq \sum_{v\in N_{G'}(z)}T(v)$ for each $i \in [k]$ and $f(k) = \sum_{v\in N_{G'}(z)}T(v)$.  Now, the following statements can easily be proven by induction on $i$ for each $i \in [0:k]$
     \begin{enumerate}[(1)]
        \item $G'_{i+1} = G'[V(G_{i+1})\cup\{z\}]$,
        \item $S'_{i+1}(v) = S_{i+1}(v)$ for all $v \in V(G_{i+1})$, $S'_{i+1}(z)=S'(z)-f(i)$, 
        
        $T'_{i+1}(v)=T_{i+1}(v)$ for all $v\in V(G_{i+1})$, and $T'_{i+1}(z)=T'(z)$, and
        \item  if $i \geq 1$, then $(s'_1, \dots, s'_i)$ is a $(G',S',T')$-legal sequence.  Furthermore, it is a restricted $(G',S',T')$-legal sequence in the case $(s_1,\ldots,s_k)$ is a restricted $(G,S,T)$-legal sequence.
    \end{enumerate}
    
    Consequently, $\mathcal{S}'$ is a $(G',S',T')$-legal sequence with $k$ terms (in the case where $G$ is also $ST^{(4)}$-$(S,T)$-degenerate, $\mathcal{S}'$ is a restricted $(G',S',T')$-legal sequence with $k$ terms) and $G'_{\mathcal{S}'}$ is the graph with vertex set $\{z\}$.  This means $G'_{\mathcal{S}'}$ is $ST^{(4)}$-$(S'_{\mathcal{S}'}, T'_{\mathcal{S}'})$-degenerate by Observation~\ref{obs:Trivial}.  Therefore, by Observation~\ref{obs:ExtendingToComplete}, $G'$ is $ST^{(3)}$-$(S',T')$-degenerate if $G$ is $ST^{(3)}$-$(S,T)$-degenerate, and $G'$ is $ST^{(4)}$-$(S',T')$-degenerate if $G$ is $ST^{(4)}$-$(S,T)$-degenerate.
\end{proof}

Finally, we obtain the following corollary of Lemma~\ref{lem:appendingVertex}, which allows us to focus only on the $k$-core of a graph in certain situations.

\begin{cor}\label{cor:k-core}
    For $k\in\N$, suppose $G$ is a graph, the $k$-core of $G$ is $ST^{(j)}$-$(s,t)$-degenerate, $j\in\{3,4\}$, and $s\geq (k-1)t$.  Then, $G$ is $ST^{(j)}$-$(s,t)$-degenerate.
\end{cor}
\begin{proof}
    Suppose $H$ is the $k$-core of $G$.  If $H=G$ the result trivially holds.  So we may suppose $H$ is a proper subgraph of $G$.  Since the $k$-core of $G$ can be obtained by successively deleting vertices of degree less than $k$, there is some $m\in\N$ and a sequence of graphs $G_0, \ldots, G_m$ such that $G_0=G$, $G_m =H$, and for each $i \in [0:m-1]$ there is a vertex $w_i$ such that $d_{G_i}(w_i) < k$ and $G_i - w_i = G_{i+1}$.   We will show that $G_{m-\mu}$ is $ST^{(j)}$-$(s,t)$-degenerate for all $\mu\in[0:m]$ by induction on $\mu$.

    The base case is clear.  So, suppose $\mu \in [m]$ and the desired result holds for all nonnegative integers less than $\mu$.  By the induction hypothesis $G_{m-\mu+1}$, is $ST^{(j)}$-$(s,t)$-degenerate. Let $S$ and $T$ be the shield and target of $G_{m-\mu}$ that are identically $s$ and identically $t$ respectively.  Since $S(w_{m-\mu}) = s \geq (k-1)t \geq |N_{G_{m-\mu}}(w_{m-\mu})|t$, Lemma~\ref{lem:appendingVertex} tells us that $G_{m-\mu}$ is $ST^{(j)}$-$(S,T)$-degenerate which completes the induction step.  Therefore, $G_{m-\mu}$ is $ST^{(j)}$-$(s,t)$-degenerate for all $\mu\in[0:m]$.  Consequently $G$ is $ST^{(j)}$-$(s,t)$-degenerate (since $G_0=G$).
\end{proof}

\section{Upper bounds}\label{sec:upperBounds}

In this section, we prove upper bounds on the fractional strict type-$4$ degeneracy of certain graphs.  Each of these results will then provide an upper bound on fractional strict type-$3$ degeneracy of the graph of interest as well as its fractional DP-chromatic number by Theorem~\ref{thm: upbound}.  We begin with a remark that gives a framework for proving that a particular sequence is a $(G,S,T)$-legal sequence with specific outputs. 

\begin{rem}\label{rem:Rigor}
    Throughout this section we often inductively construct a restricted $(G,S,T)$-legal sequence $\mathcal{S}$ where the $i^{th}$ term is of the form $s_i=\ShSa(G_i,S_i,T_i,v_i,W_i)$ in such a way that the constructed sequence has at most $n$ terms (see Remark~\ref{rem:inductiveDefinition}).  We then wish to prove that $\mathcal{S}$ has $n$ terms and is a restricted $(G,S,T)$-legal sequence.  Finally, we explicitly describe $G_\mathcal{S}$, $S_\mathcal{S}$, and $T_\mathcal{S}$, and wish to prove that these descriptions are correct.  
    
    A natural framework for proving that $\mathcal{S}$ has $n$ terms and that the outputs of $\mathcal{S}$ are correctly described is as follows.  One can prove statements about the inputs and outputs of $s_i$ by induction on $i$ that yield the following for each $i \in [0:n]$:
    \begin{enumerate}[(1)]
        \item an explicit description of the graph $G_{i+1}$,
        \item explicit formulas for $S_{i+1}$ and $T_{i+1}$, and
        \item  if $i \geq 1$, then $(s_1, \dots, s_i)$ is a restricted $(G,S,T)$-legal sequence.
    \end{enumerate}
    Explicitly utilizing this framework above often leads to unwieldy formulas and cumbersome notational details that obscure the underlying simplicity of the defined sequence.  We demonstrate this by using this framework to give all the details of the proof of Theorem~\ref{thm:c2r} in Appendix~\ref{sec:appendix}.  Subsequently, for ease of exposition, we omit this level of detail for the remaining proofs.  
\end{rem}

\subsection{Cycles}\label{subsec:cycles}

In this subsection we prove Corollary~\ref{cor:unicyclic} which states that the fractional strict type-4 degeneracy of a unicyclic graph equals its fractional chromatic number.  We begin by determining the fractional strict type-4 degeneracy of cycles. 

\begin{customthm}{\bf\ref{thm:c2r}}
	For all $r\in\N$ satisfying $r\geq2$, $ST^{(4)*}(C_{2r})=2$.  Consequently $ST^{(4)*}(C_{2r})=ST^{(3)*}(C_{2r})=\chi_{_{DP}}^*(C_{2r})=\chi_{\ell}^*(C_{2r})=\chi^*(C_{2r})=2$.
\end{customthm}

    In the proof below, we construct a restricted $(G,S,T)$-legal sequence where $G=C_{2r}$.  To aid the reader, Figure~\ref{fig:C6a} depicts an example of one such sequence when $G=C_6$.

\begin{proof}
    We have $ST^{(4)*}(C_{2r})\geq \chi^*(C_{2r})=2$.  To prove our result, we show $ST^{(4)*}(C_{2r}) \leq 2$ by showing that $C_{2r}$ is $ST^{(4)}$-$(t+1,t)$-degenerate for all $t\in\N$.  Suppose $G=C_{2r}$, and the vertices of $G$ in cyclic order are $v_1,\ldots,v_{2r}$.  Fix $t\in\N$.  For $j\in[t]$ let $S^{(j)}$ and $T^{(j)}$ be the shield and target of $G$ defined by
    \begin{equation*}
        S^{(j)}(v_i)=\left\{\begin{array}{l l}
            j&\text{if }i=1\\
            t+1&\text{if }i\text{ is even}\\
            j+1&\text{otherwise},
        \end{array}\right.
        \hspace{0.25in}\text{and}\hspace{0.25in}
        T^{(j)}(v_i)=\left\{\begin{array}{l l}
            t&\text{if }i\text{ is odd}\\
            j&\text{if }i\text{ is even}.
        \end{array}\right.
    \end{equation*}
    
    We claim that $G$ is $ST^{(4)}$-$(S^{(j)},T^{(j)})$-degenerate for all $j\in[t]$.  Assume for contradiction that this claim does not hold. Then, suppose $\mu\in[t]$ is the smallest element of $[t]$ such that $G$ is not $ST^{(4)}$-$(S^{(\mu)},T^{(\mu)})$-degenerate.  We wish to inductively construct an $r$-term restricted $(G,S^{(\mu)},T^{(\mu)})$-legal sequence.  Suppose $\mathcal{S}^{(\mu)}$ is our inductively constructed sequence of at most $r$ terms where the $i^{th}$ term is $s^{(\mu)}_i=\ShSa(v_{2i},\{v_{2i-1}\})$.

    Assume $\mu = 1$.  We show in Appendix~\ref{sec:appendix} Claim~\ref{cl:C2r1} that $\mathcal{S}^{(1)}$ contains $r$ terms (consequently $\mathcal{S}^{(1)}=(s_1^{(1)},\ldots,s_r^{(1)})$), $S^{(1)}$ is a restricted $(G,S^{(1)},T^{(1)}$-legal sequence, and $G_{\mathcal{S}^{(1)}}$ is an edgeless graph by following the framework described in Remark~\ref{rem:Rigor}.  This means $G$ is $ST^{(4)}$-$(S^{(1)},T^{(1)})$-degenerate by Observation~\ref{obs:Trivial}, which is a contradiction.  
    
    Now, assume $\mu>1$. This means that $t>1$.  We show in Appendix~\ref{sec:appendix} Claim~\ref{cl:C2r2}, following the framework described in Remark~\ref{rem:Rigor}, that $\mathcal{S}^{(\mu)}$ contains $r$ terms, $\mathcal{S}^{(\mu)}$ is a restricted $(G,S^{(\mu)},T^{(\mu)})$-legal sequence, $G_{\mathcal{S}^{(\mu)}}=C_{2r}$, 
    \begin{equation*}
        S_{\mathcal{S}^{(\mu)}}^{(\mu)}(v_i)=\left\{\begin{array}{l l}
            j-1&\text{if }i=1\\
            t+1&\text{if }i\text{ is even}\\
            j&\text{otherwise},
        \end{array}\right.
        \hspace{0.25in}\text{and}\hspace{0.25in}
        T_{\mathcal{S}^{(\mu)}}^{(\mu)}(v_i)=\left\{\begin{array}{l l}
            t&\text{if }i\text{ is odd}\\
            j-1&\text{if }i\text{ is even}.
        \end{array}\right.
    \end{equation*}
    Note that this means $S_{\mathcal{S}^{(\mu)}}^{(\mu)}=S^{(\mu-1)}$ and $T_{\mathcal{S}^{(\mu)}}^{(\mu)}=T^{(\mu-1)}$.  Due to the minimality of $\mu$ we know that $G$ is $ST^{(4)}$-$(S^{(\mu-1)},T^{(\mu-1)})$-degenerate.  So, by Observation~\ref{obs:ExtendingToComplete}, $G$ is $ST^{(4)}$-$(S^{(\mu)},T^{(\mu)})$-degenerate, which is a contradiction.

    Thus, for any $t\in\N$ and any $j\in[t]$, $G$ is $ST^{(4)}$-$(S^{(j)},T^{(j)})$-degenerate.  It follows that $G$ is $ST^{(4)}$-$(t+1,t)$-degenerate by Lemma~\ref{lem:Monotonicity} (since $S^{(t)}(v)\leq t+1$ and $T^{(t)}(v)=t$ for all $v\in V(G)$).
\end{proof}

\begin{figure}[h]
    \centering
    
    \begin{enumerate}[(1)]\centering
        
    \begin{multicols}{2}
    \item ~\hspace{1.5in}~
    \vspace{-9mm}
    
    \begin{tikzpicture}[scale=0.8,every node/.style={circle, draw, fill, inner sep=0pt, minimum size=4pt, text=black}]
			\node [anchor=center, label={[shift={(0,0.45)}]below:$v_1\hspace{-1mm}:\hspace{-1mm}(2,2)$}, color=black] (0) at (0,0) {};
			\node [anchor=center, label={[shift={(0.1,-0.45)}]above:$v_2\hspace{-1mm}:\hspace{-1mm}(3,2)$}, color=black] (1) at (1.25,0) {};
			\node [anchor=center, label={[shift={(0,0.45)}]below:$v_3\hspace{-1mm}:\hspace{-1mm}(3,2)$}, color=black] (2) at (2.5,0) {};
			\node [anchor=center, label={[shift={(0.1,-0.45)}]above:$v_4\hspace{-1mm}:\hspace{-1mm}(3,2)$}, color=black] (3) at (3.75,0) {};
			\node [anchor=center, label={[shift={(0,0.45)}]below:$v_5\hspace{-1mm}:\hspace{-1mm}(3,2)$}, color=black] (4) at (5,0) {};
			\node [anchor=center, label={[shift={(0.35,-0.45)}]above:$v_6\hspace{-1mm}:\hspace{-1mm}(3,2)$}, color=black] (5) at (6.25,0) {};

			\draw (0.center) to (1.center);
			\draw (1.center) to (2.center);
			\draw (2.center) to (3.center);
			\draw (3.center) to (4.center);
			\draw (4.center) to (5.center);
			\draw [in=45, out=135](5.center) to (0.center);
		\end{tikzpicture}
        \vspace{-5mm}
        
        \item ~\hspace{1.375in}~
    \vspace{-9mm}
    
    \begin{tikzpicture}[scale=0.8,every node/.style={circle, draw, fill, inner sep=0pt, minimum size=4pt, text=black}]
			\node [anchor=center, label={[shift={(0,0.45)}]below:$v_1\hspace{-1mm}:\hspace{-1mm}(2,2)$}, color=black] (0) at (0,0) {};
			\node [anchor=center, label={[shift={(0.1,-0.45)}]above:$v_2\hspace{-1mm}:\hspace{-1mm}(3,1)$}, color=black] (1) at (1.25,0) {};
			\node [anchor=center, label={[shift={(0,0.45)}]below:$v_3\hspace{-1mm}:\hspace{-1mm}(2,2)$}, color=black] (2) at (2.5,0) {};
			\node [anchor=center, label={[shift={(0.1,-0.45)}]above:$v_4\hspace{-1mm}:\hspace{-1mm}(3,2)$}, color=black] (3) at (3.75,0) {};
			\node [anchor=center, label={[shift={(0,0.45)}]below:$v_5\hspace{-1mm}:\hspace{-1mm}(3,2)$}, color=black] (4) at (5,0) {};
			\node [anchor=center, label={[shift={(0.35,-0.45)}]above:$v_6\hspace{-1mm}:\hspace{-1mm}(3,2)$}, color=black] (5) at (6.25,0) {};

			\draw (0.center) to (1.center);
			\draw (1.center) to (2.center);
			\draw (2.center) to (3.center);
			\draw (3.center) to (4.center);
			\draw (4.center) to (5.center);
			\draw [in=45, out=135](5.center) to (0.center);
		\end{tikzpicture}
    \end{multicols}
    \vspace{-10mm}
    
    \begin{multicols}{2}
    
        \item ~\hspace{1.375in}~
    \vspace{-9mm}
    
    \begin{tikzpicture}[scale=0.8,every node/.style={circle, draw, fill, inner sep=0pt, minimum size=4pt, text=black}]
			\node [anchor=center, label={[shift={(0,0.45)}]below:$v_1\hspace{-1mm}:\hspace{-1mm}(2,2)$}, color=black] (0) at (0,0) {};
			\node [anchor=center, label={[shift={(0.1,-0.45)}]above:$v_2\hspace{-1mm}:\hspace{-1mm}(3,1)$}, color=black] (1) at (1.25,0) {};
			\node [anchor=center, label={[shift={(0,0.45)}]below:$v_3\hspace{-1mm}:\hspace{-1mm}(2,2)$}, color=black] (2) at (2.5,0) {};
			\node [anchor=center, label={[shift={(0.1,-0.45)}]above:$v_4\hspace{-1mm}:\hspace{-1mm}(3,1)$}, color=black] (3) at (3.75,0) {};
			\node [anchor=center, label={[shift={(0,0.45)}]below:$v_5\hspace{-1mm}:\hspace{-1mm}(2,2)$}, color=black] (4) at (5,0) {};
			\node [anchor=center, label={[shift={(0.35,-0.45)}]above:$v_6\hspace{-1mm}:\hspace{-1mm}(3,2)$}, color=black] (5) at (6.25,0) {};

			\draw (0.center) to (1.center);
			\draw (1.center) to (2.center);
			\draw (2.center) to (3.center);
			\draw (3.center) to (4.center);
			\draw (4.center) to (5.center);
			\draw [in=45, out=135](5.center) to (0.center);
		\end{tikzpicture}
        \item ~\hspace{1.375in}~
    \vspace{-9mm}
    
    \begin{tikzpicture}[scale=0.8,every node/.style={circle, draw, fill, inner sep=0pt, minimum size=4pt, text=black}]
			\node [anchor=center, label={[shift={(0,0.45)}]below:$v_1\hspace{-1mm}:\hspace{-1mm}(1,2)$}, color=black] (0) at (0,0) {};
			\node [anchor=center, label={[shift={(0.1,-0.45)}]above:$v_2\hspace{-1mm}:\hspace{-1mm}(3,1)$}, color=black] (1) at (1.25,0) {};
			\node [anchor=center, label={[shift={(0,0.45)}]below:$v_3\hspace{-1mm}:\hspace{-1mm}(2,2)$}, color=black] (2) at (2.5,0) {};
			\node [anchor=center, label={[shift={(0.1,-0.45)}]above:$v_4\hspace{-1mm}:\hspace{-1mm}(3,1)$}, color=black] (3) at (3.75,0) {};
			\node [anchor=center, label={[shift={(0,0.45)}]below:$v_5\hspace{-1mm}:\hspace{-1mm}(2,2)$}, color=black] (4) at (5,0) {};
			\node [anchor=center, label={[shift={(0.35,-0.45)}]above:$v_6\hspace{-1mm}:\hspace{-1mm}(3,1)$}, color=black] (5) at (6.25,0) {};

			\draw (0.center) to (1.center);
			\draw (1.center) to (2.center);
			\draw (2.center) to (3.center);
			\draw (3.center) to (4.center);
			\draw (4.center) to (5.center);
			\draw [in=45, out=135](5.center) to (0.center);
		\end{tikzpicture}
    \end{multicols}
    \vspace{-10mm}
    
        \end{enumerate}
    \vspace{-5mm}
    
    \caption{Each (i) is a pictorial illustration of $G^{(2)}_i$ along with the ordered pairs $(S^{(2)}_i,T^{(2)}_i)$.  Notice that (2), (3), (4) correspond to outputs of the operations of $\mathcal{S}^{(2)}$ when $G=C_6$ and $t=2$ as described in Theorem~\ref{thm:c2r}.}
    \label{fig:C6a}
\end{figure}

We now present a tool that will be of great importance in the proofs of Theorems~\ref{thm:c2r+1} and~\ref{thm:subdivisions}.

\begin{lem}\label{lem:path}
    For $r\in\N$, suppose $G=P_{2r-1}$ and the vertices of $G$ in order are $v_1,\ldots,v_{2r-1}$. Suppose $S$ is the shield of $G$ defined by $S(v_1)=S(v_{2r-1})=1$ and $S(v_i)=r+1$ for all $i\in[2:2r-2]$ and $T$ is the target of $G$ defined by $T(v_i)=r$ for all $i\in[2r-1]$.  Then $G$ is $ST^{(4)}$-$(S,T)$-degenerate.
\end{lem}

    Figure~\ref{fig:P5} is an illustration of the proof below when $G=P_5$.

\begin{proof}
    If $r=1$ we see that $G = K_1$ and $S(v_1)=1$.  So, the result holds by Observation~\ref{obs:Trivial} when $r=1$.  Suppose $r\geq2$.  We will show the desired result by proving the existence of a complete, restricted $(G,S,T)$-legal sequence.

    We begin by defining a sequence of graphs, and sequences of shields and targets of those graphs.  For each $j\in[r]$ let $G^{(j)}=G$ and $S^{(j)}$ and $T^{(j)}$ be the shield and target of $G^{(j)}$ given by
    \begin{align*}
        S^{(j)}(v_i)&=\left\{\begin{array}{l l}
            (i+1)/2&\text{if }i=2k-1\text{ for some }k\in[j]\text{ and }i\neq2r-1\\
            0&\text{if }i=2r-1\text{ and }j=r\\
            1&\text{if }i=2r-1\text{ and }j\neq r\\
            r+1&\text{otherwise}
        \end{array}\right.\\
        &\text{and}\\
        T^{(j)}(v_i)&=\left\{\begin{array}{l l}
            i/2&\text{if }i=2k\text{ for some }k\in[j-1]\\
            r&\text{otherwise.}
        \end{array}\right.
    \end{align*}
    For each $j\in[r+1:2r-1]$ let $G^{(j)}=G-\{v_{2k}\}_{k\in[2r-j:r-1]}$ and $S^{(j)}$ and $T^{(j)}$ be the shield and target of $G^{(j)}$ given by
    \begin{align*}
        S^{(j)}(v_i)&=\left\{\begin{array}{l l}
            (i+1)/2&\text{if }i=2k-1\text{ for some }k\in[2r-j]\\
            0&\text{if }i=2k-1\text{ for some }k\in[2r+1-j:r]\\
            r+1&\text{otherwise}
        \end{array}\right.\\
        &\text{and}\\
        T^{(j)}(v_i)&=\left\{\begin{array}{l l}
            i/2&\text{if }i=2k\text{ for some }k\in[2r-j-1]\\
            r&\text{otherwise.}
        \end{array}\right.
    \end{align*}

    We claim that $G^{(j)}$ is $ST^{(4)}$-$(S^{(j)},T^{(j)})$-degenerate for each $j\in[2r-1]$.  This will imply the desired result since $j=1$ yields $G$ is $ST^{(4)}$-$(S,T)$-degenerate.  
    
    Assume, for contradiction, that this claim does not hold.  Then, suppose $\mu\in[2r-1]$ is the largest element of $[2r-1]$ such that $G^{(\mu)}$ is not $ST^{(4)}$-$(S^{(\mu)},T^{(\mu)})$-degenerate.  First, we obtain a contradiction in the case $\mu=2r-1$.  Note $G^{(2r-1)}=G-\{v_{2k}\}_{k\in[r-1]}$ is an edgeless graph.  Hence $G^{(2r-1)}$ is $ST^{(4)}$-$(S^{(2r-1)},T^{(2r-1)})$-degenerate by Observation~\ref{obs:Trivial}, which is a contradiction.

    Next, we obtain a contradiction when $\mu\in [r:2r-2]$.  We wish to inductively construct a $(2r-\mu-1)$-term $(G^{(\mu)},S^{(\mu)},T^{(\mu)})$-legal sequence.  Suppose $\mathcal{S}^{(\mu)}$ is our inductively constructed sequence of at most $(2r-\mu-1)$ terms where the $i^{th}$ term is $s_i^{(\mu)}=\ShSa(v_{2(2r-\mu)-2},\{v_{2(2r-\mu)-1}\})$.  Following the framework described in Remark~\ref{rem:Rigor}, it can be shown that $\mathcal{S}^{(\mu)}$ contains $(2r-\mu-1)$ terms, $\mathcal{S}^{(\mu)}$ is a restricted $(G^{(\mu)},S^{(\mu)},T^{(\mu)})$-legal sequence, $G^{(\mu)}_{\mathcal{S}^{(\mu)}}=G^{(\mu+1)}$, $S^{(\mu)}_{\mathcal{S}^{(\mu)}}=S^{(\mu+1)}$, and $T^{(\mu)}_{\mathcal{S}^{(\mu)}}=T^{(\mu+1)}$.  Due to the maximality of $\mu$ we know that $G^{(\mu+1)}$ is $ST^{(4)}$-$(S^{(\mu+1)},T^{(\mu+1)})$-degenerate.  So, by Observation~\ref{obs:ExtendingToComplete}, $G^{(\mu)}$ is $ST^{(4)}$-$(S^{(\mu)},T^{(\mu)})$-degenerate, which is a contradiction.

    Finally, we obtain a contradiction when $\mu\in [r-1]$.  We wish to inductively construct a $(r-\mu)$-term $(G^{(\mu)},S^{(\mu)},T^{(\mu)})$-legal sequence.  Suppose $\mathcal{S}^{(\mu)}$ is our inductively constructed sequence of at most $(r-\mu)$ terms where the $i^{th}$ term is $s_i^{(\mu)}=\ShSa(v_{2\mu},\{v_{2\mu-1}\})$.  Following the framework described in Remark~\ref{rem:Rigor}, it can be shown that $\mathcal{S}^{(\mu)}$ contains $(r-\mu)$ terms, $\mathcal{S}^{(\mu)}$ is a restricted $(G^{(\mu)},S^{(\mu)},T^{(\mu)})$-legal sequence, $G^{(\mu)}_{\mathcal{S}^{(\mu)}}=G^{(\mu+1)}$, $S^{(\mu)}_{\mathcal{S}^{(\mu)}}=S^{(\mu+1)}$, and $T^{(\mu)}_{\mathcal{S}^{(\mu)}}=T^{(\mu+1)}$.  Due to the maximality of $\mu$ we know that $G^{(\mu+1)}$ is $ST^{(4)}$-$(S^{(\mu+1)},T^{(\mu+1)})$-degenerate.  So, by Observation~\ref{obs:ExtendingToComplete}, $G^{(\mu)}$ is $ST^{(4)}$-$(S^{(\mu)},T^{(\mu)})$-degenerate, which is a contradiction.

\end{proof}

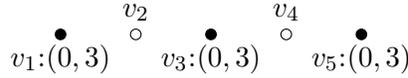
\begin{figure}[h!]
    \begin{itemize}

        \item If $\mu=1$:
        \vspace{-1.5cm}
    
    \begin{center}\begin{tikzpicture}[scale=0.8,every node/.style={circle, draw, fill, inner sep=0pt, minimum size=4pt, text=black}]
			\node [anchor=center, label={[shift={(0,0.45)}]below:$v_1\hspace{-1mm}:\hspace{-1mm}(1,3)$}, color=black] (0) at (-1.25,0) {};
			\node [anchor=center, label={[shift={(0,-0.45)}]above:$v_2\hspace{-1mm}:\hspace{-1mm}(4,3)$}, color=black] (1) at (0,0) {};
			\node [anchor=center, label={[shift={(0,0.45)}]below:$v_3\hspace{-1mm}:\hspace{-1mm}(4,3)$}, color=black] (2) at (1.25,0) {};
			\node [anchor=center, label={[shift={(0,-0.45)}]above:$v_4\hspace{-1mm}:\hspace{-1mm}(4,3)$}, color=black] (3) at (2.5,0) {};
			\node [anchor=center, label={[shift={(0,0.45)}]below:$v_5\hspace{-1mm}:\hspace{-1mm}(1,3)$}, color=black] (4) at (3.75,0) {};

			\draw (0.center) to (1.center);
			\draw (1.center) to (2.center);
			\draw (2.center) to (3.center);
			\draw (3.center) to (4.center);

            \draw [->, thick] (-0.5,-0.75) to (1.5,-1.5);
            \node [anchor=center, label={center:$s^{(1)}_1=\ShSa(v_2,\{v_1\})$}, color=white] (5) at (-2.5,-1.5) {};

			\node [anchor=center, label={[shift={(0,0.45)}]below:$v_1\hspace{-1mm}:\hspace{-1mm}(1,3)$}, color=black] (0) at (2,-1.75) {};
			\node [anchor=center, label={[shift={(0,-0.45)}]above:$v_2\hspace{-1mm}:\hspace{-1mm}(4,2)$}, color=black] (1) at (3.25,-1.75) {};
			\node [anchor=center, label={[shift={(0,0.45)}]below:$v_3\hspace{-1mm}:\hspace{-1mm}(3,3)$}, color=black] (2) at (4.5,-1.75) {};
			\node [anchor=center, label={[shift={(0,-0.45)}]above:$v_4\hspace{-1mm}:\hspace{-1mm}(4,3)$}, color=black] (3) at (5.75,-1.75) {};
			\node [anchor=center, label={[shift={(0,0.45)}]below:$v_5\hspace{-1mm}:\hspace{-1mm}(1,3)$}, color=black] (4) at (7,-1.75) {};

			\draw (0.center) to (1.center);
			\draw (1.center) to (2.center);
			\draw (2.center) to (3.center);
			\draw (3.center) to (4.center);

            \draw [->, thick] (2.75,-2.5) to (4.75,-3.25);
            \node [anchor=center, label={center:$s^{(1)}_2=\ShSa(v_2,\{v_1\})$}, color=white] (5) at (0.75,-3.25) {};

			\node [anchor=center, label={[shift={(0,0.45)}]below:$v_1\hspace{-1mm}:\hspace{-1mm}(1,3)$}, color=black] (0) at (5.25,-3.5) {};
			\node [anchor=center, label={[shift={(0,-0.45)}]above:$v_2\hspace{-1mm}:\hspace{-1mm}(4,1)$}, color=black] (1) at (6.5,-3.5) {};
			\node [anchor=center, label={[shift={(0,0.45)}]below:$v_3\hspace{-1mm}:\hspace{-1mm}(2,3)$}, color=black] (2) at (7.75,-3.5) {};
			\node [anchor=center, label={[shift={(0,-0.45)}]above:$v_4\hspace{-1mm}:\hspace{-1mm}(4,3)$}, color=black] (3) at (9,-3.5) {};
			\node [anchor=center, label={[shift={(0,0.45)}]below:$v_5\hspace{-1mm}:\hspace{-1mm}(1,3)$}, color=black] (4) at (10.25,-3.5) {};

			\draw (0.center) to (1.center);
			\draw (1.center) to (2.center);
			\draw (2.center) to (3.center);
			\draw (3.center) to (4.center);
		\end{tikzpicture}\end{center}
        \vspace{-1.75cm}
    
        \item If $\mu=2$:
        \vspace{-1.5cm}
    
    \begin{center}\begin{tikzpicture}[scale=0.8,every node/.style={circle, draw, fill, inner sep=0pt, minimum size=4pt, text=black}]
			\node [anchor=center, label={[shift={(0,0.45)}]below:$v_1\hspace{-1mm}:\hspace{-1mm}(1,3)$}, color=black] (0) at (-1.25,0) {};
			\node [anchor=center, label={[shift={(0,-0.45)}]above:$v_2\hspace{-1mm}:\hspace{-1mm}(4,1)$}, color=black] (1) at (0,0) {};
			\node [anchor=center, label={[shift={(0,0.45)}]below:$v_3\hspace{-1mm}:\hspace{-1mm}(2,3)$}, color=black] (2) at (1.25,0) {};
			\node [anchor=center, label={[shift={(0,-0.45)}]above:$v_4\hspace{-1mm}:\hspace{-1mm}(4,3)$}, color=black] (3) at (2.5,0) {};
			\node [anchor=center, label={[shift={(0,0.45)}]below:$v_5\hspace{-1mm}:\hspace{-1mm}(1,3)$}, color=black] (4) at (3.75,0) {};

			\draw (0.center) to (1.center);
			\draw (1.center) to (2.center);
			\draw (2.center) to (3.center);
			\draw (3.center) to (4.center);

            \draw [->, thick] (-0.5,-0.75) to (1.5,-1.5);
            \node [anchor=center, label={center:$s^{(2)}_1=\ShSa(v_4,\{v_3\})$}, color=white] (5) at (-2.5,-1.5) {};

			\node [anchor=center, label={[shift={(0,0.45)}]below:$v_1\hspace{-1mm}:\hspace{-1mm}(1,3)$}, color=black] (0) at (2,-1.75) {};
			\node [anchor=center, label={[shift={(0,-0.45)}]above:$v_2\hspace{-1mm}:\hspace{-1mm}(4,1)$}, color=black] (1) at (3.25,-1.75) {};
			\node [anchor=center, label={[shift={(0,0.45)}]below:$v_3\hspace{-1mm}:\hspace{-1mm}(2,3)$}, color=black] (2) at (4.5,-1.75) {};
			\node [anchor=center, label={[shift={(0,-0.45)}]above:$v_4\hspace{-1mm}:\hspace{-1mm}(4,2)$}, color=black] (3) at (5.75,-1.75) {};
			\node [anchor=center, label={[shift={(0,0.45)}]below:$v_5\hspace{-1mm}:\hspace{-1mm}(0,3)$}, color=black] (4) at (7,-1.75) {};

			\draw (0.center) to (1.center);
			\draw (1.center) to (2.center);
			\draw (2.center) to (3.center);
			\draw (3.center) to (4.center);
		\end{tikzpicture}\end{center}
        \vspace{-1.75cm}
    
        \item If $\mu=3$:
        \vspace{-1.5cm}
    
    \begin{center}\begin{tikzpicture}[scale=0.8,every node/.style={circle, draw, fill, inner sep=0pt, minimum size=4pt, text=black}]
			\node [anchor=center, label={[shift={(0,0.45)}]below:$v_1\hspace{-1mm}:\hspace{-1mm}(1,3)$}, color=black] (0) at (-1.25,0) {};
			\node [anchor=center, label={[shift={(0,-0.45)}]above:$v_2\hspace{-1mm}:\hspace{-1mm}(4,1)$}, color=black] (1) at (0,0) {};
			\node [anchor=center, label={[shift={(0,0.45)}]below:$v_3\hspace{-1mm}:\hspace{-1mm}(2,3)$}, color=black] (2) at (1.25,0) {};
			\node [anchor=center, label={[shift={(0,-0.45)}]above:$v_4\hspace{-1mm}:\hspace{-1mm}(4,2)$}, color=black] (3) at (2.5,0) {};
			\node [anchor=center, label={[shift={(0,0.45)}]below:$v_5\hspace{-1mm}:\hspace{-1mm}(0,3)$}, color=black] (4) at (3.75,0) {};

			\draw (0.center) to (1.center);
			\draw (1.center) to (2.center);
			\draw (2.center) to (3.center);
			\draw (3.center) to (4.center);

            \draw [->, thick] (-0.5,-0.75) to (1.5,-1.5);
            \node [anchor=center, label={center:$s^{(3)}_1=\ShSa(v_4,\{v_5\})$}, color=white] (5) at (-2.5,-1.5) {};

			\node [anchor=center, label={[shift={(0,0.45)}]below:$v_1\hspace{-1mm}:\hspace{-1mm}(1,3)$}, color=black] (0) at (2,-1.75) {};
			\node [anchor=center, label={[shift={(0,-0.45)}]above:$v_2\hspace{-1mm}:\hspace{-1mm}(4,1)$}, color=black] (1) at (3.25,-1.75) {};
			\node [anchor=center, label={[shift={(0,0.45)}]below:$v_3\hspace{-1mm}:\hspace{-1mm}(1,3)$}, color=black] (2) at (4.5,-1.75) {};
			\node [anchor=center, label={[shift={(0,-0.45)}]above:$v_4\hspace{-1mm}:\hspace{-1mm}(4,1)$}, color=black] (3) at (5.75,-1.75) {};
			\node [anchor=center, label={[shift={(0,0.45)}]below:$v_5\hspace{-1mm}:\hspace{-1mm}(0,3)$}, color=black] (4) at (7,-1.75) {};

			\draw (0.center) to (1.center);
			\draw (1.center) to (2.center);
			\draw (2.center) to (3.center);
			\draw (3.center) to (4.center);

            \draw [->, thick] (2.75,-2.5) to (4.75,-3.25);
            \node [anchor=center, label={center:$s^{(3)}_2=\ShSa(v_4,\{v_5\})$}, color=white] (5) at (0.75,-3.25) {};

			\node [anchor=center, label={[shift={(0,0.45)}]below:$v_1\hspace{-1mm}:\hspace{-1mm}(1,3)$}, color=black] (0) at (5.25,-3.5) {};
			\node [anchor=center, label={[shift={(0,-0.45)}]above:$v_2\hspace{-1mm}:\hspace{-1mm}(4,1)$}, color=black] (1) at (6.5,-3.5) {};
			\node [anchor=center, label={[shift={(0,0.45)}]below:$v_3\hspace{-1mm}:\hspace{-1mm}(0,3)$}, color=black] (2) at (7.75,-3.5) {};
			\node [anchor=center, label={[shift={(0,0)}]above:$v_4$}, color=black, fill=white] (3) at (9,-3.5) {};
			\node [anchor=center, label={[shift={(0,0.45)}]below:$v_5\hspace{-1mm}:\hspace{-1mm}(0,3)$}, color=black] (4) at (10.25,-3.5) {};

			\draw (0.center) to (1.center);
			\draw (1.center) to (2.center);
		\end{tikzpicture}\end{center}
        \vspace{-1.75cm}
    
        \item If $\mu=4$:
        \vspace{-1.5cm}
    
    \begin{center}\begin{tikzpicture}[scale=0.8,every node/.style={circle, draw, fill, inner sep=0pt, minimum size=4pt, text=black}]
			\node [anchor=center, label={[shift={(0,0.45)}]below:$v_1\hspace{-1mm}:\hspace{-1mm}(1,3)$}, color=black] (0) at (-1.25,0) {};
			\node [anchor=center, label={[shift={(0,-0.45)}]above:$v_2\hspace{-1mm}:\hspace{-1mm}(4,1)$}, color=black] (1) at (0,0) {};
			\node [anchor=center, label={[shift={(0,0.45)}]below:$v_3\hspace{-1mm}:\hspace{-1mm}(0,3)$}, color=black] (2) at (1.25,0) {};
			\node [anchor=center, label={[shift={(0,0)}]above:$v_4$}, color=black, fill=white] (3) at (2.5,0) {};
			\node [anchor=center, label={[shift={(0,0.45)}]below:$v_5\hspace{-1mm}:\hspace{-1mm}(0,3)$}, color=black] (4) at (3.75,0) {};

			\draw (0.center) to (1.center);
			\draw (1.center) to (2.center);

            \draw [->, thick] (-0.5,-0.75) to (1.5,-1.5);
            \node [anchor=center, label={center:$s^{(4)}_1=\ShSa(v_2,\{v_3\})$}, color=white] (5) at (-2.5,-1.5) {};

			\node [anchor=center, label={[shift={(0,0.45)}]below:$v_1\hspace{-1mm}:\hspace{-1mm}(0,3)$}, color=black] (0) at (2,-1.75) {};
			\node [anchor=center, label={[shift={(0,0)}]above:$v_2$}, color=black, fill=white] (1) at (3.25,-1.75) {};
			\node [anchor=center, label={[shift={(0,0.45)}]below:$v_3\hspace{-1mm}:\hspace{-1mm}(0,3)$}, color=black] (2) at (4.5,-1.75) {};
			\node [anchor=center, label={[shift={(0,0)}]above:$v_4$}, color=black, fill=white] (3) at (5.75,-1.75) {};
			\node [anchor=center, label={[shift={(0,0.45)}]below:$v_5\hspace{-1mm}:\hspace{-1mm}(0,3)$}, color=black] (4) at (7,-1.75) {};
		\end{tikzpicture}\end{center}
        \vspace{-1.75cm}
        
        \item If $\mu=5$:
        \vspace{-1.5cm}
    
    \begin{center}\begin{tikzpicture}[scale=0.8,every node/.style={circle, draw, fill, inner sep=0pt, minimum size=4pt, text=black}]
			\node [anchor=center, label={[shift={(0,0.45)}]below:$v_1\hspace{-1mm}:\hspace{-1mm}(0,3)$}, color=black] (0) at (-1.25,0) {};
			\node [anchor=center, label={[shift={(0,0)}]above:$v_2$}, color=black, fill=white] (1) at (0,0) {};
			\node [anchor=center, label={[shift={(0,0.45)}]below:$v_3\hspace{-1mm}:\hspace{-1mm}(0,3)$}, color=black] (2) at (1.25,0) {};
			\node [anchor=center, label={[shift={(0,0)}]above:$v_4$}, color=black, fill=white] (3) at (2.5,0) {};
			\node [anchor=center, label={[shift={(0,0.45)}]below:$v_5\hspace{-1mm}:\hspace{-1mm}(0,3)$}, color=black] (4) at (3.75,0) {};

		\end{tikzpicture}\end{center}
        \vspace{-1.25cm}
        
    \end{itemize}
    
    \caption{In the proof of Lemma~\ref{lem:path} when $G=P_5$ we know that $\mu\in[5]$.  In the case $\mu=5$, $G^{(5)}$ is edgeless.  For each of the other possible cases, the figure provides pictorial illustrations of $G^{(\mu)}_i$ along with the ordered pairs $(S^{(\mu)}_i,T^{(\mu)}_i)$ for the operations of $\mathcal{S}^{(\mu)}$.}
    \label{fig:P5}
\end{figure}

\begin{customthm}{\bf\ref{thm:c2r+1}}
	For all $r\in\N$, $ST^{(4)*}(C_{2r+1})=2+1/r$.  Consequently $ST^{(4)*}(C_{2r+1})=ST^{(3)*}(C_{2r+1})=\chi_{_{DP}}^*(C_{2r+1})=\chi_{\ell}^*(C_{2r+1})=\chi^*(C_{2r+1})=2+1/r$.
\end{customthm}

    In the proof below, we construct an appropriate restricted $(G,S,T)$-legal sequence where $G=C_{2r+1}$.  To aid the reader, Figure~\ref{fig:C7} depicts this sequence when $G=C_7$.

\begin{proof}
    It is well known that $\chi^*(C_{2r+1})=2+1/r$ (see e.g.,~\cite{AT97}) which means $ST^{(4)*}(C_{2r+1})\geq 2+1/r$.  To prove our result, we show $ST^{(4)*}(C_{2r+1}) \leq 2 + 1/r$ by showing that $C_{2r+1}$ is $ST^{(4)}$-$(r+1,r)$-degenerate.  We accomplish this by showing the existence of an appropriate restricted $(G,S,T)$-legal sequence whose final output graph, shield, and target satisfy the hypothesis of Lemma~\ref{lem:path}.

    Suppose $G=C_{2r+1}$ and the vertices of $G$ in cyclic order are $v_1,\ldots,v_{2r+1}$. Suppose $S$ and $T$ are the shield and target of $G$ whose values are identically $r+1$ and $r$ respectively.  We wish to inductively construct a $2r$-term restricted $(G,S,T)$-legal sequence.  Suppose $\mathcal{S}$ is our inductively constructed sequence of at most $2r$ terms where the $i^{th}$ term is given by $s_i=\Sh(v_{2r})$ for $i\in[r]$ and $s_i=\Sh(v_{2r+1})$ for $i\in[r+1:2r]$.

    Following the framework described in Remark~\ref{rem:Rigor}, it can be shown that $\mathcal{S}$ contains $2r$ terms and $\mathcal{S}$ is a restricted $(G,S,T)$-legal sequence.  It can also be shown that $G_\mathcal{S}=P_{2r-1}$, $S_\mathcal{S}(v_1)=S_\mathcal{S}(v_{2r-1})=1$ while $S_\mathcal{S}(v_i)=r+1$ for all $i\in[2:2r-2]$, and $T_\mathcal{S}(v_i)=r$ for all $i\in[2r-1]$.
    
    Notice that $G_\mathcal{S}$, $S_\mathcal{S}$ and $T_\mathcal{S}$ satisfy the hypothesis of Lemma~\ref{lem:path}.  Consequently $G_\mathcal{S}$ is $ST^{(4)}$-$(S_\mathcal{S},T_\mathcal{S})$-degenerate.  Therefore, $G$ is $ST^{(4)}$-$(r+1,r)$-degenerate by Observation~\ref{obs:ExtendingToComplete}.
\end{proof}

\begin{figure}[h]
    \centering
    
    \begin{enumerate}[(1)]\centering
    \item ~\hspace{1.375in}~
    \vspace{-11mm}
    
    \begin{tikzpicture}[scale=0.8,every node/.style={circle, draw, fill, inner sep=0pt, minimum size=4pt, text=black}]
			\node [anchor=center, label={[shift={(0,0.45)}]below:$v_1\hspace{-1mm}:\hspace{-1mm}(4,3)$}, color=black] (0) at (-1.125,0) {};
			\node [anchor=center, label={[shift={(0.15,-0.45)}]above:$v_2\hspace{-1mm}:\hspace{-1mm}(4,3)$}, color=black] (1) at (0,0) {};
			\node [anchor=center, label={[shift={(0,0.45)}]below:$v_3\hspace{-1mm}:\hspace{-1mm}(4,3)$}, color=black] (2) at (1.125,0) {};
			\node [anchor=center, label={[shift={(0,-0.45)}]above:$v_4\hspace{-1mm}:\hspace{-1mm}(4,3)$}, color=black] (3) at (2.25,0) {};
			\node [anchor=center, label={[shift={(0,0.45)}]below:$v_5\hspace{-1mm}:\hspace{-1mm}(4,3)$}, color=black] (4) at (3.375,0) {};
			\node [anchor=center, label={[shift={(-0.15,-0.45)}]above:$v_6\hspace{-1mm}:\hspace{-1mm}(4,3)$}, color=black] (5) at (4.5,0) {};
			\node [anchor=center, label={[shift={(0,0.45)}]below:$v_7\hspace{-1mm}:\hspace{-1mm}(4,3)$}, color=black] (6) at (5.625,0) {};

			\draw (0.center) to (1.center);
			\draw (1.center) to (2.center);
			\draw (2.center) to (3.center);
			\draw (3.center) to (4.center);
			\draw (4.center) to (5.center);
			\draw (5.center) to (6.center);
            \draw [in=45, out=135](6.center) to (0.center);
		\end{tikzpicture}
        \vspace{-8mm}
        
    \begin{multicols}{2}
        \item ~\hspace{1.375in}~
    \vspace{-11mm}
    
    \begin{tikzpicture}[scale=0.8,every node/.style={circle, draw, fill, inner sep=0pt, minimum size=4pt, text=black}]
			\node [anchor=center, label={[shift={(0,0.45)}]below:$v_1\hspace{-1mm}:\hspace{-1mm}(4,3)$}, color=black] (0) at (-1.125,0) {};
			\node [anchor=center, label={[shift={(0.15,-0.45)}]above:$v_2\hspace{-1mm}:\hspace{-1mm}(4,3)$}, color=black] (1) at (0,0) {};
			\node [anchor=center, label={[shift={(0,0.45)}]below:$v_3\hspace{-1mm}:\hspace{-1mm}(4,3)$}, color=black] (2) at (1.125,0) {};
			\node [anchor=center, label={[shift={(0,-0.45)}]above:$v_4\hspace{-1mm}:\hspace{-1mm}(4,3)$}, color=black] (3) at (2.25,0) {};
			\node [anchor=center, label={[shift={(0,0.45)}]below:$v_5\hspace{-1mm}:\hspace{-1mm}(3,3)$}, color=black] (4) at (3.375,0) {};
			\node [anchor=center, label={[shift={(-0.15,-0.45)}]above:$v_6\hspace{-1mm}:\hspace{-1mm}(4,2)$}, color=black] (5) at (4.5,0) {};
			\node [anchor=center, label={[shift={(0,0.45)}]below:$v_7\hspace{-1mm}:\hspace{-1mm}(3,3)$}, color=black] (6) at (5.625,0) {};

			\draw (0.center) to (1.center);
			\draw (1.center) to (2.center);
			\draw (2.center) to (3.center);
			\draw (3.center) to (4.center);
			\draw (4.center) to (5.center);
			\draw (5.center) to (6.center);
            \draw [in=45, out=135](6.center) to (0.center);
		\end{tikzpicture}
    
        \item ~\hspace{1.375in}~
    \vspace{-11mm}
    
    \begin{tikzpicture}[scale=0.8,every node/.style={circle, draw, fill, inner sep=0pt, minimum size=4pt, text=black}]
			\node [anchor=center, label={[shift={(0,0.45)}]below:$v_1\hspace{-1mm}:\hspace{-1mm}(4,3)$}, color=black] (0) at (-1.125,0) {};
			\node [anchor=center, label={[shift={(0.15,-0.45)}]above:$v_2\hspace{-1mm}:\hspace{-1mm}(4,3)$}, color=black] (1) at (0,0) {};
			\node [anchor=center, label={[shift={(0,0.45)}]below:$v_3\hspace{-1mm}:\hspace{-1mm}(4,3)$}, color=black] (2) at (1.125,0) {};
			\node [anchor=center, label={[shift={(0,-0.45)}]above:$v_4\hspace{-1mm}:\hspace{-1mm}(4,3)$}, color=black] (3) at (2.25,0) {};
			\node [anchor=center, label={[shift={(0,0.45)}]below:$v_5\hspace{-1mm}:\hspace{-1mm}(2,3)$}, color=black] (4) at (3.375,0) {};
			\node [anchor=center, label={[shift={(-0.15,-0.45)}]above:$v_6\hspace{-1mm}:\hspace{-1mm}(4,1)$}, color=black] (5) at (4.5,0) {};
			\node [anchor=center, label={[shift={(0,0.45)}]below:$v_7\hspace{-1mm}:\hspace{-1mm}(2,3)$}, color=black] (6) at (5.625,0) {};

			\draw (0.center) to (1.center);
			\draw (1.center) to (2.center);
			\draw (2.center) to (3.center);
			\draw (3.center) to (4.center);
			\draw (4.center) to (5.center);
			\draw (5.center) to (6.center);
            \draw [in=45, out=135](6.center) to (0.center);
		\end{tikzpicture}
    \end{multicols}
    \vspace{-12mm}
        
    \begin{multicols}{2}

        \item ~\hspace{1.375in}~
    \vspace{-11mm}
    
    \begin{tikzpicture}[scale=0.8,every node/.style={circle, draw, fill, inner sep=0pt, minimum size=4pt, text=black}]
			\node [anchor=center, label={[shift={(0,0.45)}]below:$v_1\hspace{-1mm}:\hspace{-1mm}(4,3)$}, color=black] (0) at (-1.125,0) {};
			\node [anchor=center, label={[shift={(0.15,-0.45)}]above:$v_2\hspace{-1mm}:\hspace{-1mm}(4,3)$}, color=black] (1) at (0,0) {};
			\node [anchor=center, label={[shift={(0,0.45)}]below:$v_3\hspace{-1mm}:\hspace{-1mm}(4,3)$}, color=black] (2) at (1.125,0) {};
			\node [anchor=center, label={[shift={(0,-0.45)}]above:$v_4\hspace{-1mm}:\hspace{-1mm}(4,3)$}, color=black] (3) at (2.25,0) {};
			\node [anchor=center, label={[shift={(0,0.45)}]below:$v_5\hspace{-1mm}:\hspace{-1mm}(1,3)$}, color=black] (4) at (3.375,0) {};
			\node [anchor=center, label={[shift={(0,0)}]above:$v_6$}, color=black, fill=white] (5) at (4.5,0) {};
			\node [anchor=center, label={[shift={(0,0.45)}]below:$v_7\hspace{-1mm}:\hspace{-1mm}(1,3)$}, color=black] (6) at (5.625,0) {};

			\draw (0.center) to (1.center);
			\draw (1.center) to (2.center);
			\draw (2.center) to (3.center);
			\draw (3.center) to (4.center);
            \draw [in=45, out=135](6.center) to (0.center);
		\end{tikzpicture}

        \item ~\hspace{1.375in}~
    \vspace{-11mm}
    
    \begin{tikzpicture}[scale=0.8,every node/.style={circle, draw, fill, inner sep=0pt, minimum size=4pt, text=black}]
			\node [anchor=center, label={[shift={(0,0.45)}]below:$v_1\hspace{-1mm}:\hspace{-1mm}(3,3)$}, color=black] (0) at (-1.125,0) {};
			\node [anchor=center, label={[shift={(0.15,-0.45)}]above:$v_2\hspace{-1mm}:\hspace{-1mm}(4,3)$}, color=black] (1) at (0,0) {};
			\node [anchor=center, label={[shift={(0,0.45)}]below:$v_3\hspace{-1mm}:\hspace{-1mm}(4,3)$}, color=black] (2) at (1.125,0) {};
			\node [anchor=center, label={[shift={(0,-0.45)}]above:$v_4\hspace{-1mm}:\hspace{-1mm}(4,3)$}, color=black] (3) at (2.25,0) {};
			\node [anchor=center, label={[shift={(0,0.45)}]below:$v_5\hspace{-1mm}:\hspace{-1mm}(1,3)$}, color=black] (4) at (3.375,0) {};
			\node [anchor=center, label={[shift={(0,0)}]above:$v_6$}, color=black, fill=white] (5) at (4.5,0) {};
			\node [anchor=center, label={[shift={(0,0.45)}]below:$v_7\hspace{-1mm}:\hspace{-1mm}(1,2)$}, color=black] (6) at (5.625,0) {};

			\draw (0.center) to (1.center);
			\draw (1.center) to (2.center);
			\draw (2.center) to (3.center);
			\draw (3.center) to (4.center);
            \draw [in=45, out=135](6.center) to (0.center);
		\end{tikzpicture}
    \end{multicols}
    \vspace{-12mm}
        
    \begin{multicols}{2}
    
        \item ~\hspace{1.375in}~
    \vspace{-11mm}
    
    \begin{tikzpicture}[scale=0.8,every node/.style={circle, draw, fill, inner sep=0pt, minimum size=4pt, text=black}]
			\node [anchor=center, label={[shift={(0,0.45)}]below:$v_1\hspace{-1mm}:\hspace{-1mm}(2,3)$}, color=black] (0) at (-1.125,0) {};
			\node [anchor=center, label={[shift={(0.15,-0.45)}]above:$v_2\hspace{-1mm}:\hspace{-1mm}(4,3)$}, color=black] (1) at (0,0) {};
			\node [anchor=center, label={[shift={(0,0.45)}]below:$v_3\hspace{-1mm}:\hspace{-1mm}(4,3)$}, color=black] (2) at (1.125,0) {};
			\node [anchor=center, label={[shift={(0,-0.45)}]above:$v_4\hspace{-1mm}:\hspace{-1mm}(4,3)$}, color=black] (3) at (2.25,0) {};
			\node [anchor=center, label={[shift={(0,0.45)}]below:$v_5\hspace{-1mm}:\hspace{-1mm}(1,3)$}, color=black] (4) at (3.375,0) {};
			\node [anchor=center, label={[shift={(0,0)}]above:$v_6$}, color=black, fill=white] (5) at (4.5,0) {};
			\node [anchor=center, label={[shift={(0,0.45)}]below:$v_7\hspace{-1mm}:\hspace{-1mm}(1,1)$}, color=black] (6) at (5.625,0) {};

			\draw (0.center) to (1.center);
			\draw (1.center) to (2.center);
			\draw (2.center) to (3.center);
			\draw (3.center) to (4.center);
            \draw [in=45, out=135](6.center) to (0.center);
		\end{tikzpicture}

        \item ~\hspace{1.375in}~
    \vspace{-11mm}
    
    \begin{tikzpicture}[scale=0.8,every node/.style={circle, draw, fill, inner sep=0pt, minimum size=4pt, text=black}]
			
            \draw [in=45, out=135, color=white](6.center) to (0.center);\node [anchor=center, label={[shift={(0,0.45)}]below:$v_1\hspace{-1mm}:\hspace{-1mm}(1,3)$}, color=black] (0) at (-1.125,0) {};
			\node [anchor=center, label={[shift={(0.15,-0.45)}]above:$v_2\hspace{-1mm}:\hspace{-1mm}(4,3)$}, color=black] (1) at (0,0) {};
			\node [anchor=center, label={[shift={(0,0.45)}]below:$v_3\hspace{-1mm}:\hspace{-1mm}(4,3)$}, color=black] (2) at (1.125,0) {};
			\node [anchor=center, label={[shift={(0,-0.45)}]above:$v_4\hspace{-1mm}:\hspace{-1mm}(4,3)$}, color=black] (3) at (2.25,0) {};
			\node [anchor=center, label={[shift={(0,0.45)}]below:$v_5\hspace{-1mm}:\hspace{-1mm}(1,3)$}, color=black] (4) at (3.375,0) {};
			\node [anchor=center, label={[shift={(0,0)}]above:$v_6$}, color=black, fill=white] (5) at (4.5,0) {};
			\node [anchor=center, label={[shift={(0,0)}]below:$v_7$}, color=black, fill=white] (6) at (5.625,0) {};

			\draw (0.center) to (1.center);
			\draw (1.center) to (2.center);
			\draw (2.center) to (3.center);
			\draw (3.center) to (4.center);
		\end{tikzpicture}
    \end{multicols}
    \vspace{-10mm}
        \end{enumerate}
    \vspace{-6mm}
    
    \caption{Each (i) is a pictorial illustration of $G_i$ along with the ordered pairs $(S_i,T_i)$.  Notice that (2), (3), (4), (5), (6), and (7) correspond to the outputs of the operations of $\mathcal{S}$ when $G=C_7$ as described in the proof of Theorem~\ref{thm:c2r+1}.}
    \label{fig:C7}
\end{figure}
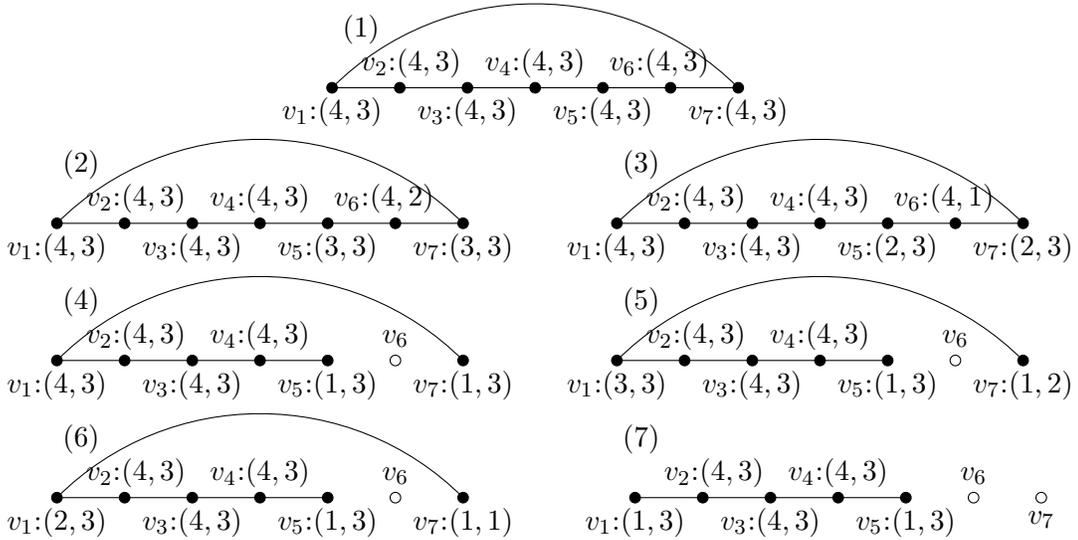

Recall that Zhou, Zhu, and Zhu~\cite{ZZZ23} showed that the Alon-Tarsi number of a graph $G$ is a lower bound on $ST^{(1)}(G)$.  So, we have $AT(G)\leq ST^{(3)}(G)$. Notice Theorem~\ref{thm:c2r+1} shows that the Alon-Tarsi number of a graph need not be a lower bound on its fractional strict type-4 degeneracy.  Indeed, for $r\geq2$, $ST^{(4)*}(C_{2r+1})<AT(C_{2r+1}) = 3$. 

We end this subsection by using Theorems~\ref{thm:c2r} and~\ref{thm:c2r+1} in conjunction with Corollary~\ref{cor:k-core} to prove that the fractional strict type-4 degeneracy of a unicyclic graph equals its fractional chromatic number.  With this final corollary and the results above we will have that connected graphs $G$ with average degree at most 2 satisfy $ST^{(4)*}(G)=\chi_{_{DP}}^*(G)$.  Recall that this means such graphs satisfy the equation in Question~\ref{ques:slack}.

\begin{cor}\label{cor:unicyclic}
    For any unicyclic graph $G$, $ST^{(4)*}(G)=\chi^*(G)$.
\end{cor}
\begin{proof}

    Suppose $H$ is the cycle contained in $G$.  Due to the monotonicity of fractional coloring is easy to show that $\chi^*(H)\leq\chi^*(G)$.  It is also clear that $H$ is the $2$-core of $G$.  We will prove the desired result when $H$ is an odd cycle and when $H$ is an even cycle.

    First, suppose $H=C_{2r+1}$ for some $r\in\N$.  By Theorem~\ref{thm:c2r+1} we have that $H$ is $ST^{(4)}$-$(r+1,r)$-degenerate.  Since $r+1\geq (2-1)r$, we know that $G$ is $ST^{(4)}$-$(r+1,r)$-degenerate by Corollary~\ref{cor:k-core}.  Thus, $2+1/r = \chi^*(H) \leq \chi^*(G) \leq ST^{(4)*}(G) \leq 2+1/r$ as desired.
    
    Next, suppose $H=C_{2r}$ for some $r\in\N$ satisfying $r\geq2$.  By the proof of Theorem~\ref{thm:c2r} we have that $H$ is $ST^{(4)}$-$(t+1,t)$-degenerate for all $t\in\N$.  Since $t+1\geq (2-1)t$, we know that $G$ is $ST^{(4)}$-$(t+1,t)$-degenerate by Corollary~\ref{cor:k-core} for all $t \in \N$.  Thus, $2 = \chi^*(H) \leq \chi^*(G) \leq ST^{(4)*}(G) \leq 2$ as desired.
\end{proof}

\subsection{Complete bipartite graphs}\label{subsec:CompBipGraphs}

We now prove an upper bound on the fractional strict type-$4$ degeneracy of complete bipartite graphs.  The following theorem gives us improvement on some of the previously best known upper bounds on certain fractional DP-chromatic numbers.  For example, Theorem~\ref{thm:Km,nUpperBound} shows that $ST^{(4)*}(K_{2,3})\leq7/3$.  Consequently, Theorem~\ref{thm: upbound} shows that $\chi^*_{_{DP}}(K_{2,3})\leq7/3$.  This is an improvement of the previously best known upper bound of approximately of $2.619$, shown in~\cite{DKM25}.

\begin{customthm}{\bf\ref{thm:Km,nUpperBound}}
	For all $m,n\in\N$ satisfying $2\leq m\leq n$, $ST^{(4)*}(K_{m,n})\leq m+1-m/n$.  Consequently $\chi_{_{DP}}^*(K_{m,n})\leq m+1-m/n$.
\end{customthm}

    In the proof below, we construct a restricted $(G,S,T)$-legal sequence where $G=K_{m,n}$.  To aid the reader, Figure~\ref{fig:K23_9,6} depicts an example of one such sequence when $G=K_{2,3}$.

\begin{proof}
    Suppose $G=K_{m,n}$ and the partite sets of $G$ are $A=\{a_1,\ldots,a_m\}$ and $B=\{b_1,\ldots,b_n\}$.   We will show that $G$ is $ST^{(4)}$-$(m(n-1)t+1,nt)$-degenerate for all $t\in\N$ which implies $ST^{(4)*}(K_{m,n})\leq m+1-m/n$. We achieve this by proving something stronger.

    Fix $t\in\N$. For each $j\in[t]$, suppose $S^{(j)}$ is the shield of $G$ given by
    \begin{equation*}
        S^{(j)}(v)=\left\{\begin{array}{l l}m(n-1)t+1&\text{if }v\in A \\m(n-1)j&\text{if }v\in B,\end{array}\right.
    \end{equation*}
    and $T^{(j)}$ is the target of $G$ given by
    \begin{equation*}
        T^{(j)}(v)=\left\{\begin{array}{l l}nj&\text{if }v\in A\\
        nt&\text{if }v\in B.\end{array}\right.
    \end{equation*}
    We claim that $G$ is $ST^{(4)}$-$(S^{(j)},T^{(j)})$-degenerate for all $j\in[t]$.  Assume for contradiction that this claim does not hold.  Then suppose $\mu\in[t]$ is the smallest element of $[t]$ such that $G$ is not $ST^{(4)}$-$(S^{(\mu)},T^{(\mu)})$-degenerate.  We wish to inductively construct an $mn$-term restricted $(G,S^{(\mu)},T^{(\mu)})$-legal sequence.  Suppose $\mathcal{S}^{(\mu)}$ is our inductively constructed sequence of up to $mn$ terms where the $i^{th}$ term is $s_i^{(\mu)}=\ShSa(a_{q+1},\{b_{r+1}\})$ for all $i\in[mn]$ where $i-1=qn+r$ for some $q\in\{0,\ldots,m-1\}$ and $r\in\{0,\ldots,n-1\}$.

    Assume $\mu=1$.  Following the framework described in Remark~\ref{rem:Rigor}, it can be shown that $\mathcal{S}^{(1)}$ has $mn$ terms, $\mathcal{S}^{(1)}$ is a restricted $(G,S^{(1)},T^{(1)})$-legal sequence, and $G_{\mathcal{S}^{(\mu)}}$ is an edgeless graph.  This means $G$ is $ST^{(4)}$-$(S^{(1)},T^{(1)})$-degenerate by Observation~\ref{obs:Trivial}, which is a contradiction.

    Now, assume $\mu>1$.  Following the framework described in Remark~\ref{rem:Rigor}, it can be shown that $\mathcal{S}^{(\mu)}$ has $mn$ terms, $\mathcal{S}^{(\mu)}$ is a restricted $(G,S^{(\mu)},T^{(\mu)})$-legal sequence, $G_{\mathcal{S}^{(\mu)}}=G$,
    \begin{equation*}
        S^{(\mu)}_{\mathcal{S}^{(\mu)}}(v)=\left\{\begin{array}{l l}m(n-1)t+1&\text{if }v\in A \\m(n-1)(\mu-1)&\text{if }v\in B\end{array}\right.,
        \hspace{0.25in}\text{and}\hspace{0.25in}
        T^{(\mu)}_{\mathcal{S}^{(\mu)}}(v)=\left\{\begin{array}{l l}n(\mu-1)&\text{if }v\in A\\
        nt&\text{if }v\in B.\end{array}\right.
    \end{equation*}
    Note that this means $S^{(\mu)}_{\mathcal{S}^{(\mu)}}=S^{(\mu-1)}$ and $T^{(\mu)}_{\mathcal{S}^{(\mu)}}=T^{(\mu-1)}$.  Due to the minimality of $\mu$ we know that $G$ is $ST^{(4)}$-$(S^{(\mu-1)},T^{(\mu-1)})$-degenerate.  This implies, by Observation~\ref{obs:ExtendingToComplete}, that $G$ is $ST^{(4)}$-$(S^{(\mu)},T^{(\mu)})$-degenerate, which is a contradiction.

    Thus, for any $t\in\N$ and any $j\in[t]$, $G$ is $ST^{(4)}$-$(S^{(j)},T^{(j)})$-degenerate.  It follows that $G$ is $ST^{(4)}$-$(m(n-1)t+1,nt)$-degenerate by Lemma~\ref{lem:Monotonicity} for all $t\in\N$.  The desired result follows.
\end{proof}

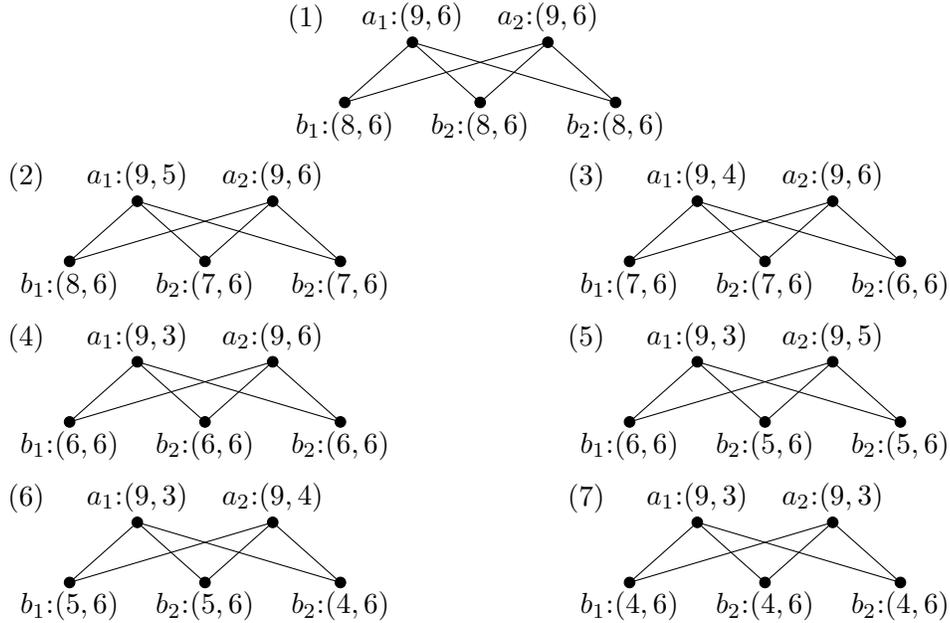
\begin{figure}[h]
    \centering
    
    \begin{enumerate}[(1)]\centering
    \item ~\hspace{1.375in}~
    \vspace{-11mm}
    
    \begin{tikzpicture}[scale=0.8,every node/.style={circle, draw, fill, inner sep=0pt, minimum size=4pt, text=black}]
			\node [anchor=center, label={[shift={(0,-0.45)}]above:$a_1\hspace{-1mm}:\hspace{-1mm}(9,6)$}, color=black] (0) at (0,1) {};
			\node [anchor=center, label={[shift={(0,-0.45)}]above:$a_2\hspace{-1mm}:\hspace{-1mm}(9,6)$}, color=black] (1) at (2.25,1) {};
			\node [anchor=center, label={[shift={(0,0.45)}]below:$b_1\hspace{-1mm}:\hspace{-1mm}(8,6)$}, color=black] (2) at (-1.125,0) {};
			\node [anchor=center, label={[shift={(0,0.45)}]below:$b_2\hspace{-1mm}:\hspace{-1mm}(8,6)$}, color=black] (3) at (1.125,0) {};
			\node [anchor=center, label={[shift={(0,0.45)}]below:$b_2\hspace{-1mm}:\hspace{-1mm}(8,6)$}, color=black] (4) at (3.375,0) {};

			\draw (0.center) to (2.center);
			\draw (0.center) to (3.center);
			\draw (0.center) to (4.center);
			\draw (1.center) to (2.center);
			\draw (1.center) to (3.center);
			\draw (1.center) to (4.center);
		\end{tikzpicture}
        \vspace{-8mm}
        
    \begin{multicols}{2}
        \item ~\hspace{1.375in}~
    \vspace{-11mm}

        \begin{tikzpicture}[scale=0.8,every node/.style={circle, draw, fill, inner sep=0pt, minimum size=4pt, text=black}]
			\node [anchor=center, label={[shift={(0,-0.45)}]above:$a_1\hspace{-1mm}:\hspace{-1mm}(9,5)$}, color=black] (0) at (0,1) {};
			\node [anchor=center, label={[shift={(0,-0.45)}]above:$a_2\hspace{-1mm}:\hspace{-1mm}(9,6)$}, color=black] (1) at (2.25,1) {};
			\node [anchor=center, label={[shift={(0,0.45)}]below:$b_1\hspace{-1mm}:\hspace{-1mm}(8,6)$}, color=black] (2) at (-1.125,0) {};
			\node [anchor=center, label={[shift={(0,0.45)}]below:$b_2\hspace{-1mm}:\hspace{-1mm}(7,6)$}, color=black] (3) at (1.125,0) {};
			\node [anchor=center, label={[shift={(0,0.45)}]below:$b_2\hspace{-1mm}:\hspace{-1mm}(7,6)$}, color=black] (4) at (3.375,0) {};

			\draw (0.center) to (2.center);
			\draw (0.center) to (3.center);
			\draw (0.center) to (4.center);
			\draw (1.center) to (2.center);
			\draw (1.center) to (3.center);
			\draw (1.center) to (4.center);
		\end{tikzpicture}
    
        \item ~\hspace{1.375in}~
    \vspace{-11mm}

        \begin{tikzpicture}[scale=0.8,every node/.style={circle, draw, fill, inner sep=0pt, minimum size=4pt, text=black}]
			\node [anchor=center, label={[shift={(0,-0.45)}]above:$a_1\hspace{-1mm}:\hspace{-1mm}(9,4)$}, color=black] (0) at (0,1) {};
			\node [anchor=center, label={[shift={(0,-0.45)}]above:$a_2\hspace{-1mm}:\hspace{-1mm}(9,6)$}, color=black] (1) at (2.25,1) {};
			\node [anchor=center, label={[shift={(0,0.45)}]below:$b_1\hspace{-1mm}:\hspace{-1mm}(7,6)$}, color=black] (2) at (-1.125,0) {};
			\node [anchor=center, label={[shift={(0,0.45)}]below:$b_2\hspace{-1mm}:\hspace{-1mm}(7,6)$}, color=black] (3) at (1.125,0) {};
			\node [anchor=center, label={[shift={(0,0.45)}]below:$b_2\hspace{-1mm}:\hspace{-1mm}(6,6)$}, color=black] (4) at (3.375,0) {};

			\draw (0.center) to (2.center);
			\draw (0.center) to (3.center);
			\draw (0.center) to (4.center);
			\draw (1.center) to (2.center);
			\draw (1.center) to (3.center);
			\draw (1.center) to (4.center);
		\end{tikzpicture}
    \end{multicols}
    \vspace{-12mm}
        
    \begin{multicols}{2}

        \item ~\hspace{1.375in}~
    \vspace{-11mm}

        \begin{tikzpicture}[scale=0.8,every node/.style={circle, draw, fill, inner sep=0pt, minimum size=4pt, text=black}]
			\node [anchor=center, label={[shift={(0,-0.45)}]above:$a_1\hspace{-1mm}:\hspace{-1mm}(9,3)$}, color=black] (0) at (0,1) {};
			\node [anchor=center, label={[shift={(0,-0.45)}]above:$a_2\hspace{-1mm}:\hspace{-1mm}(9,6)$}, color=black] (1) at (2.25,1) {};
			\node [anchor=center, label={[shift={(0,0.45)}]below:$b_1\hspace{-1mm}:\hspace{-1mm}(6,6)$}, color=black] (2) at (-1.125,0) {};
			\node [anchor=center, label={[shift={(0,0.45)}]below:$b_2\hspace{-1mm}:\hspace{-1mm}(6,6)$}, color=black] (3) at (1.125,0) {};
			\node [anchor=center, label={[shift={(0,0.45)}]below:$b_2\hspace{-1mm}:\hspace{-1mm}(6,6)$}, color=black] (4) at (3.375,0) {};

			\draw (0.center) to (2.center);
			\draw (0.center) to (3.center);
			\draw (0.center) to (4.center);
			\draw (1.center) to (2.center);
			\draw (1.center) to (3.center);
			\draw (1.center) to (4.center);
		\end{tikzpicture}

        \item ~\hspace{1.375in}~
    \vspace{-11mm}

        \begin{tikzpicture}[scale=0.8,every node/.style={circle, draw, fill, inner sep=0pt, minimum size=4pt, text=black}]
			\node [anchor=center, label={[shift={(0,-0.45)}]above:$a_1\hspace{-1mm}:\hspace{-1mm}(9,3)$}, color=black] (0) at (0,1) {};
			\node [anchor=center, label={[shift={(0,-0.45)}]above:$a_2\hspace{-1mm}:\hspace{-1mm}(9,5)$}, color=black] (1) at (2.25,1) {};
			\node [anchor=center, label={[shift={(0,0.45)}]below:$b_1\hspace{-1mm}:\hspace{-1mm}(6,6)$}, color=black] (2) at (-1.125,0) {};
			\node [anchor=center, label={[shift={(0,0.45)}]below:$b_2\hspace{-1mm}:\hspace{-1mm}(5,6)$}, color=black] (3) at (1.125,0) {};
			\node [anchor=center, label={[shift={(0,0.45)}]below:$b_2\hspace{-1mm}:\hspace{-1mm}(5,6)$}, color=black] (4) at (3.375,0) {};

			\draw (0.center) to (2.center);
			\draw (0.center) to (3.center);
			\draw (0.center) to (4.center);
			\draw (1.center) to (2.center);
			\draw (1.center) to (3.center);
			\draw (1.center) to (4.center);
		\end{tikzpicture}
    \end{multicols}
    \vspace{-12mm}
        
    \begin{multicols}{2}
    
        \item ~\hspace{1.375in}~
    \vspace{-11mm}

        \begin{tikzpicture}[scale=0.8,every node/.style={circle, draw, fill, inner sep=0pt, minimum size=4pt, text=black}]
			\node [anchor=center, label={[shift={(0,-0.45)}]above:$a_1\hspace{-1mm}:\hspace{-1mm}(9,3)$}, color=black] (0) at (0,1) {};
			\node [anchor=center, label={[shift={(0,-0.45)}]above:$a_2\hspace{-1mm}:\hspace{-1mm}(9,4)$}, color=black] (1) at (2.25,1) {};
			\node [anchor=center, label={[shift={(0,0.45)}]below:$b_1\hspace{-1mm}:\hspace{-1mm}(5,6)$}, color=black] (2) at (-1.125,0) {};
			\node [anchor=center, label={[shift={(0,0.45)}]below:$b_2\hspace{-1mm}:\hspace{-1mm}(5,6)$}, color=black] (3) at (1.125,0) {};
			\node [anchor=center, label={[shift={(0,0.45)}]below:$b_2\hspace{-1mm}:\hspace{-1mm}(4,6)$}, color=black] (4) at (3.375,0) {};

			\draw (0.center) to (2.center);
			\draw (0.center) to (3.center);
			\draw (0.center) to (4.center);
			\draw (1.center) to (2.center);
			\draw (1.center) to (3.center);
			\draw (1.center) to (4.center);
		\end{tikzpicture}

        \item ~\hspace{1.375in}~
    \vspace{-11mm}

        \begin{tikzpicture}[scale=0.8,every node/.style={circle, draw, fill, inner sep=0pt, minimum size=4pt, text=black}]
			\node [anchor=center, label={[shift={(0,-0.45)}]above:$a_1\hspace{-1mm}:\hspace{-1mm}(9,3)$}, color=black] (0) at (0,1) {};
			\node [anchor=center, label={[shift={(0,-0.45)}]above:$a_2\hspace{-1mm}:\hspace{-1mm}(9,3)$}, color=black] (1) at (2.25,1) {};
			\node [anchor=center, label={[shift={(0,0.45)}]below:$b_1\hspace{-1mm}:\hspace{-1mm}(4,6)$}, color=black] (2) at (-1.125,0) {};
			\node [anchor=center, label={[shift={(0,0.45)}]below:$b_2\hspace{-1mm}:\hspace{-1mm}(4,6)$}, color=black] (3) at (1.125,0) {};
			\node [anchor=center, label={[shift={(0,0.45)}]below:$b_2\hspace{-1mm}:\hspace{-1mm}(4,6)$}, color=black] (4) at (3.375,0) {};

			\draw (0.center) to (2.center);
			\draw (0.center) to (3.center);
			\draw (0.center) to (4.center);
			\draw (1.center) to (2.center);
			\draw (1.center) to (3.center);
			\draw (1.center) to (4.center);
		\end{tikzpicture}
    \end{multicols}
    \vspace{-10mm}
        \end{enumerate}
    \vspace{-6mm}
    
    \caption{Each (i) is a pictorial illustration of $G^{(2)}_i$ along with the ordered pairs $(S^{(2)}_i,T^{(2)}_i)$.  Notice that (2), (3), (4), (5), (6), and (7) correspond to the outputs of the operations of $\mathcal{S}^{(2)}$ when $G=K_{2,3}$ as described in the proof of Theorem~\ref{thm:Km,nUpperBound}.}
    \label{fig:K23_9,6}
\end{figure}

As stated before, Theorem~\ref{thm:Km,nUpperBound} leads to a few of the best known upper bounds on certain fractional DP-chromatic numbers.  Recall that based on the results from~\cite{DKM25}, $\chi_{_{DP}}^*(K_{2,n})\leq2.619$ for all $n\in\N$.  In addition to showing $\chi_{_{DP}}^*(K_{2,3})\leq7/3$, a simple application of Theorem~\ref{thm:Km,nUpperBound} also shows $\chi_{_{DP}}^*(K_{2,4})\leq5/2$, and $\chi_{_{DP}}^*(K_{2,5})\leq13/5$.

\subsection{Subdivisions}\label{subsec:subdivisions}

Theorem~\ref{thm: fracDP2} says that for a connected graph $G$, $\chi_{_{DP}}^*(G)=2$ if and only if $G$ contains a single even cycle and no odd cycles.  Consequently, any connected graph containing more than one cycle has fractional DP-chromatic number strictly greater than $2$. In this subsection we prove Theorem~\ref{thm:subdivisions} which shows the existence of connected graphs with many cycles whose fractional DP-chromatic numbers are close to $2$.

Our strategy for proving Theorem~\ref{thm:subdivisions} is as follows.  First, we prove a corollary to Lemma~\ref{lem:path} about the fractional strict type-$4$ degeneracy of paths.  The final ingredient of the proof of Theorem~\ref{thm:subdivisions} is Lemma~\ref{lem:RedToLinForest} which may be of independent interest.

\begin{cor}\label{cor:paths}
	For $r\in\N$ and $k\geq2r-1$, suppose $G=P_{k}$ and the vertices of $G$ in order are $v_1,\ldots,v_{k}$. Suppose $S$ is the shield of $G$ defined by $S(v_1)=S(v_{k})=1$ and $S(v_i)=r+1$ for all $i\in[2:k-1]$ and $T$ is the target of $G$ defined by $T(v_i)=r$ for all $i\in[k]$.  Then $G$ is $ST^{(4)}$-$(S,T)$-degenerate.
\end{cor}

\begin{proof}
    Notice that when $r=1$, $S(v_i) \geq d_G(v_i)$  while $T(v_i)=1$ for all $i\in[k]$.  So, when $r=1$, the desired result follows immediately from Observation~\ref{obs:DegenerateByBackDegree}.  Therefore, we may assume $r\geq2$.

    For a fixed $r$, we claim that our result holds for all $k\geq2r-1$.  Assume for contradiction that this claim does not hold.  Then, suppose $\mu$ is the smallest nonnegative integer such that $G$ is not $ST^{(4)}$-$(S,T)$-degenerate when $k-2r+1=\mu$.  Assume $\mu=0$.  When $\mu=0$ we see that $k=2r-1$ and $G=P_{2r-1}$ while $S$ and $T$ satisfy the hypotheses of Lemma~\ref{lem:path}.  Therefore, $G$ is $ST^{(4)}$-$(S,T)$-degenerate by Lemma~\ref{lem:path}, which is a contradiction.

    Now, assume $\mu>1$.  We wish to inductively construct an $r$-term restricted $(G,S,T)$-legal sequence.  Suppose $\mathcal{S}$ is our inductively constructed sequence of up to $r$ terms where the $i^{th}$ term is $s_i=\Sh(v_k)$.  Following the framework described in Remark~\ref{rem:Rigor}, it can be shown that $\mathcal{S}$ has $r$ terms, $\mathcal{S}$ is a restricted $(G,S,T)$-legal sequence, $G_\mathcal{S}=G[V(G)\setminus\{v_k\}]=P_{k-1}$, $S_\mathcal{S}(v_1)=S_\mathcal{S}(v_{k-1})=1$ and $S_\mathcal{S}(v_i)=r+1$ for all $i\in[2:k-2]$ and $T_\mathcal{S}(v_i)=r$ for all $i\in[k-1]$.  Since $\mu$ is minimal and $(k-1)-2r+1=\mu-1$, we know that $G_\mathcal{S}$ is $ST^{(4)}$-$(S_\mathcal{S},T_\mathcal{S})$-degenerate.  This along with Observation~\ref{obs:ExtendingToComplete} implies that $G$ is $ST^{(4)}$-$(S,T)$-degenerate, which is a contradiction.
\end{proof}

We're now ready to prove Lemma~\ref{lem:RedToLinForest}.

\begin{lem}\label{lem:RedToLinForest}
	Suppose $G$ is a graph and $X$ is an independent set of vertices in $G$.  If $X$ contains $\{v\in V(G):d_G(v)\geq3\}$, $G-X$ is a linear forest, and the order of each component of $G-X$ is at least $2r-1$ for some $r\in\N$, then $G$ is $ST^{(4)}$-$(r+1,r)$-degenerate.  Consequently $\chi^*_{_{DP}}(G)\leq ST^{(4)*}(G)\leq2+1/r$.
\end{lem}

\begin{proof}
    First, note that when $X = \emptyset$, $G$ is a linear forest which means $G$ is 1-degenerate.  This means $G$ is $ST^{(4)}$-$(1,1)$-degenerate, then $G$ is $ST^{(4)}$-$(r,r)$-degenerate by Theorem~\ref{thm:ks,kt-degenerate}.  Finally, $G$ is $ST^{(4)}$-$(r+1,r)$-degenerate by Lemma~\ref{lem:Monotonicity}.  So, we may assume that $X$ is nonempty.
    
    Suppose that the components of $G-X$ are $W_1, \ldots, W_q$.  Since $G-X$ is a linear forest we know that these components are paths.  Let $$E = \{v \in V(G) : \text{$v$ is an endpoint of a component of $G-X$}\}.$$  Notice that $d_G(v)\leq2$ for all $v\in E$.  Consider the partition of $E$, $\{A,B,C\}$ where $A$ consists of all elements of $E$ adjacent to two vertices in $X$, $B$ consists of all elements of $E$ adjacent to one vertex in $X$, and $C$ consists of all elements of $E$ that are not adjacent to any vertex in $X$.  Finally, note that when $r \geq 2$, $A = \emptyset$.
    
    Suppose $S$ and $T$ are the shield and target of $G$ that are identically $r+1$ and $r$ respectively.  Additionally, suppose that the independent set $X$ has $m$ elements and those elements are $x_1,\ldots,x_m$.  We wish to construct an $mr$-term restricted $(G,S,T)$-legal sequence.  Suppose $\mathcal{S}$ is our inductively constructed sequence of at most $mr$ terms where the $i^{th}$ term is $s_i=\Sh(x_{\lceil i/r\rceil})$ (i.e., each of the vertices in $X$ is shaded $r$ times).

    Following the framework described in Remark~\ref{rem:Rigor}, it can be shown that $\mathcal{S}$ has $mr$ terms, $\mathcal{S}$ is a restricted $(G,S,T)$-legal sequence, $G_{\mathcal{S}}=G-X$, $T_{\mathcal{S}}$ is the target of $G_{\mathcal{S}}$ given by
    \begin{equation*}
        T_{\mathcal{S}}(v)=r,
    \end{equation*} and $S_\mathcal{S}$ is the shield of $G_\mathcal{S}$ given by
    \begin{equation*}
        S_\mathcal{S}(v)=\left\{\begin{array}{l l}
            0&\text{if }v\in A\\
            1&\text{if }v\in B\\
            r+1&\text{otherwise.}\\
        \end{array}\right.
    \end{equation*} 
    To see why $S_\mathcal{S}(v)=0$ when $v \in A$, recall that it is only possible for $A\neq\emptyset$ when $r=1$.  

    Now, let $S'$ be the shield of $G-X$ given by
    \begin{equation*}
        S'(v)=\left\{\begin{array}{l l}
            0&\text{if }v\in A\\
            1&\text{if }v\in B\cup C\\
            r+1&\text{otherwise.}\\
        \end{array}\right.
    \end{equation*} 
    Notice that $S'(v)\leq S_{\mathcal{S}}(v)$ for all $v\in V(G-X)$.  If we can show that $G-X$ is $ST^{(4)}$-$(S',T_\mathcal{S})$-degenerate, then it will follow that $G-X$ is $ST^{(4)}$-$(S_\mathcal{S},T_\mathcal{S})$-degenerate  by Lemma~\ref{lem:Monotonicity}. Then $G$ will be $ST^{(4)}$-$(r+1,r)$-degenerate by Observation~\ref{obs:ExtendingToComplete}.
    
    The components of $G-X$ are paths.  Suppose $S^{(i)}$ and $T^{(i)}$ are the restrictions of $S'$ and $T_\mathcal{S}$ respectively to $V(W_i)$ for each $i\in[q]$.  In the case that $W_i=P_1$, $W_i$ is $ST^{(4)}$-$(S^{(i)},T^{(i)})$-degenerate by Observation~\ref{obs:Trivial}.  In the case that $W_i$ is a path with two or more vertices, $W_i$ is $ST^{(4)}$-$(S^{(i)},T^{(i)})$-degenerate by Corollary~\ref{cor:paths}.  Consequently $G-X$ is $ST^{(4)}$-$(S',T_\mathcal{S})$-degenerate by Observation~\ref{obs:Components}.  This completes the proof.
\end{proof}

It is worth noting here that the condition on $X$ in Lemma~\ref{lem:RedToLinForest} is a restricted form of the Independent Feedback Vertex Set problem that has been well studied in the literature beginning in 2012 with Misra, Philip, Raman, and Saurabh~\cite{MPRS}.  In addition to the requirements for an Independent Feedback Vertex Set $X$ of the graph $G$: (1) $G-X$ is a forest and (2) $X$ is an independent set in $G$; Lemma~\ref{lem:RedToLinForest} adds the requirements that $G-X$ is a linear forest with sufficiently long components and that every vertex of degree $3$ or more in $G$ is contained in $X$.

Lemma~\ref{lem:RedToLinForest} can now be used to prove Theorem~\ref{thm:subdivisions}.

\begin{customthm}{\bf\ref{thm:subdivisions}} 
    Suppose $H$ is a connected graph and $G$ is a subdivision of $H$ where all of the edges are replaced by internally disjoint paths each of whose lengths is at least $2r$ for some $r\geq2$.  Then $\chi^*_{_{DP}}(G)\leq ST^{(4)*}(G)\leq2+1/r$.
\end{customthm}
\begin{proof}

    Suppose $X=V(H)$.  Since $G$ is a subdivision of $H$ where all of the edges are replaced by internally disjoint paths of length at least $2r$ for some $r\geq2$, it is clear that $X$ is an independent set in $G$.  Additionally, since every vertex of $G-X$ is one of the internal vertices of the paths that replaced the edges of $H$, $d_G(v)\leq2$ for all $v\in V(G-X)$.  Consequently, $\{v\in V(G):d_G(v)\geq3\}\subseteq X$, $G-X$ is a linear forest, and the order of each component of $G-X$ is at least $2r-1$ for some $r\in\N$.  Therefore, by Lemma~\ref{lem:RedToLinForest}, $G$ is $ST^{(4)}$-$(r+1,r)$-degenerate.  This means $\chi^*_{_{DP}}(G)\leq ST^{(4)*}(G)\leq2+1/r$.
\end{proof}

Recall Theorem~\ref{thm: fracDP2} tells us that any connected graph with more than one cycle has fractional DP-chromatic number strictly greater than $2$.  Theorem~\ref{thm:subdivisions} makes it easy to construct connected graphs with many cycles whose fractional DP-chromatic numbers are close to $2$.

\section{Lower bounds for complete bipartite graphs}\label{sec:lowerBound}

We now consider lower bounds on the fractional strict type-$4$ degeneracy of complete bipartite graphs.  Also, recall that a lower bound on the fractional strict type-$3$ or strict type-$4$ degeneracy does not provide a lower bound for the fractional DP-chromatic number (so it has less utility when studying coloring problems).  Therefore, a lower bound gives us the best possible upper bound we can hope to obtain for the fractional DP-chromatic number of a graph via its fractional degeneracy.

Interestingly, in an old version of~\cite{DKM25} written by Kaul and the second named author~\cite{KM20v1}, it is asked whether $\chi_{_{DP}}^*(K_{m,n})$ can be made arbitrarily close to $m+1$ for sufficiently large $n$.  While we know the answer to this is no in a strong sense (see Corollary~$7$ in~\cite{DKM25}),  Theorem~\ref{thm:CompleteBipartiteLowerBound} shows that the answer to the analogue of the question in the setting of strict type-4 degeneracy is yes.

\begin{customthm}{\bf\ref{thm:CompleteBipartiteLowerBound}}
	For any $\epsilon\in(0,1]$, $ST^{(4)*}(K_{m,n})\geq m+1-\epsilon$ whenever $m\in\N$ and $n\geq 5m/\epsilon$.
\end{customthm}

\begin{proof}
    
    Let $G=K_{m,n}$ where $n\geq 5m/\epsilon$.  Since $\chi^*_{_{DP}}(K_{1,n})\geq\chi^*(K_{1,n})=2$ for all $n\geq1$, the result clearly holds when $m=1$.  So we may assume $m\geq2$ for the remainder of the proof.  Assume for the sake of contradiction that $ST^{(4)*}(G)< m+1-\epsilon$.  This implies that there are $s,t\in\N$ such that $s/t<m-\epsilon$ and there exists a complete, restricted $(G,s,t)$-legal sequence $\mathcal{S}=(s_1,\ldots,s_{(m+n)t})$.  Suppose the partite sets of $G$ are $A=\{a_1,\ldots,a_m\}$ and $B=\{b_1,\ldots,b_n\}$, and assume that the vertices in each partite set are named so that if $a_j$ is removed by $s_{j'}$ and $a_\ell$ is removed by $s_{\ell'}$, then $j<\ell$ if and only if $j'>\ell'$, and likewise for the vertices in $B$. Note $a_m$ is the first vertex in $A$ that is removed by an operation in $\mathcal{S}$, $a_1$ is the last vertex in $A$ that is removed, and likewise for the vertices in $B$.  For each $i\in[(m+n)t+1]$, let $A_i=A\cap V(G_i)$, $\alpha_i=|A_i|$, $B_i=B\cap V(G_i)$, and $\beta_i=|B_i|$.

    In order to obtain a contradiction, we first verify three claims. In Claim $1$ we give an upper bound on the number of operations in $\mathcal{S}$ that shade a vertex in $B$ and have a graphical input that contains the vertices $a_1$ and $a_2$.  In Claim $2$ we give a lower bound on the number of operations that shade a vertex in $B$ and have a graphical input that contains the vertex $a_1$.  In Claim $3$, we give an upper bound on the number of vertices $b\in B$ such that for some $i$ where $A_i=\{a_1\}$ we have $b\in B_i$ and $S_i(b)\geq T_i(a_1)$.
    
    The intuitive idea behind the rest of the proof is as follows.  We use Claims $1$ and $2$ to obtain a lower bound on the number of operations that shade a vertex in $B$ and have a graphical input that contains $a_1$ but not $a_2$.  This, when combined with Claim $3$, provides a lower bound on the number of operations that shade a vertex in $B$ and do not save on $a_1$ while the graphical input still contains $a_1$.  This allows us to show $S_j(a_1)<0$ for some $j\in[(m+n)t+1]$, which is a contradiction.

    For the remainder of the proof, suppose $q$ is the smallest index such that $\alpha_{q}=1$ and $r$ is the smallest index such that $\alpha_{r}=0$.

    \noindent {\bf Claim $1$:}  It is true that
    \begin{equation*}
       \sum_{b\in B_{q}}T_{q}(b)\geq\frac{3mt}{\epsilon}.
    \end{equation*}
    We know $\alpha_{q-1}=2$. Notice each operation in $(s_1,\ldots,s_{q-1})$ that shades some $b \in B$ can save on at most one of $a_1$ and $a_2$.  Since $S_1(a_1)+S_1(a_2)=2s$, there can be at most $2s<2mt\leq 2mt/\epsilon$ operations in $(s_1,\ldots,s_{q-1})$ that shade some $b\in B$.  Therefore,
    \begin{equation*}
        \sum_{b\in B_q}T_{q}(b) \geq nt - 2s \geq\frac{5mt}{\epsilon}-2mt\geq\frac{3mt}{\epsilon}.
    \end{equation*}
    This completes the proof of Claim $1$.

    \nobreak\noindent {\bf Claim $2$:} It is true that $\beta_r<m/\epsilon$.  Consequently for any $i\in[r]$
    \begin{equation*}
        \sum_{b\in B_{r}}T_{i}(b)< \frac{mt}{\epsilon}\hspace{0.25in}\text{and specifically}\hspace{0.25in}\sum_{b\in B_{r}}T_{q}(b)< \frac{mt}{\epsilon}.
    \end{equation*}
    Since each $\ShSa$ operation that shades any $a\in A$ in $(s_1,\ldots,s_{r-1})$ can save on at most one vertex in $B_{r}$ and there have been $mt$ operations in $(s_1,\ldots,s_{r-1})$ that shade a vertex in $A$, we see
    \begin{equation*}
        \sum_{b\in B_r}S(b)-(\beta_r-1)mt\geq\sum_{b\in B_r}S_r(b)\geq0,
    \end{equation*}
    and hence
    \begin{equation*}
        \sum_{b\in B_r}S(b)-(\beta_r-1)mt=\beta_r s-(\beta_r-1)mt\geq0.
    \end{equation*}
    So,
    \begin{equation*}
        (\beta_{r}-1)mt\leq\beta_{r}s<\beta_{r}(m-\epsilon)t.
    \end{equation*}
    This means that $-mt<-\beta_{r}\epsilon t$ and consequently $\beta_{r}<m/\epsilon$.  This completes the proof of Claim $2$.

    \noindent{\bf Claim $3$:} Suppose $B^*=\{b:b\in B_{i}\text{ and }S_{i}(b)\geq T_{i}(a_1)\text{ for some }i\in[q:r-1]\}$, then
    \begin{equation*}
        |B^*|\leq \frac{m}{\epsilon}\text{, and consequently }\sum_{b\in B^*}T_{q}(b)\leq\frac{mt}{\epsilon}.
    \end{equation*}

    Let $\Sa(v,j)=|\{s:s\in\{s_1,\ldots,s_j\}\text{ and }s\text{ saves on }v\}|$ for each $j \in [(m+n)t]$ and $v \in V(G)$.  Assume for the sake of contradiction that $|B^*|>m/\epsilon$.  Consider an arbitrary $b \in B^*$.  We know there is a $j\in[q:r-1]$ such that $S_j(b)\geq T_j(a_1)$. Since there are $(m-1)t$ operations in $(s_1,\ldots,s_{j-1})$ that shade vertices in $\{a_2,\ldots,a_m\}$ and there are $t-T_{j}(a_1)$ operations in $(s_1,\ldots,s_{j-1})$ that shade $a_1$,
    \begin{equation*}
        S_j(b)=s-(m-1)t-(t-T_j(a_1))+\Sa(b,j-1)=s-mt+T_j(a_1)+\Sa(b,j-1).
    \end{equation*}
    Consequently, $S_j(b)\geq T_j(a_1)$ implies $s-mt+T_j(a_1)+\Sa(b,j-1) \geq T_j(a_1)$.  Then, since $s/t<m-\epsilon$, 
    \begin{equation*}
        \Sa(b,j-1)\geq mt-s>mt-(m-\epsilon)t=t\epsilon.
    \end{equation*}
    
    Since each operation in $\mathcal{S}$ can save on at most one vertex, this means there must be at least $t\epsilon|B^*|$ operations in $\mathcal{S}$ that save on a vertex in $B^*$.  Consequently, there are more than $mt$ operations that save on a vertex in $B^*$ since $t\epsilon|B^*|>t\epsilon(m/\epsilon)=mt$. However, there are $mt$ operations in $\mathcal{S}$ that shade a vertex in $A$. Since operations that shade vertices in $A$ are the only operations that can save on a vertex in $B^*$, we have a contradiction which completes the proof of Claim $3$.

    Let $\Delta=B_q-(B^*\cup B_r)$, $\Gamma = B_r$, and $\Lambda = B^* - B_r$.  Since $B^*\subseteq B_q$ and $B_r\subseteq B_q$ we have that $\{\Delta, \Gamma, \Lambda \}$ is a partition of $B_q$.  By Claims $1$, $2$, and $3$, we know
    \begin{equation*}
        \sum_{b\in \Delta}T_{q}(b)+\sum_{b\in \Gamma}T_{q}(b)+\sum_{b\in \Lambda}T_{q}(b)\geq\frac{3mt}{\epsilon},\hspace{0.4in}
        \sum_{b\in \Gamma}T_{q}(b)<\frac{mt}{\epsilon},\text{ and}\hspace{0.4in}
        \sum_{b\in \Lambda}T_{q}(b)\leq\frac{mt}{\epsilon},
    \end{equation*}
    which implies
    \begin{equation}\label{ineq:sumInDelta}
        \sum_{b\in \Delta}T_{q}(b)>\frac{mt}{\epsilon}.
    \end{equation}
    Let $\tilde{\mathcal{S}}$ be the set of all operations in $\{s_q,\ldots,s_{r-1}\}$ that shade a vertex in $\Delta$.  We will consider two cases: there is an $s\in\tilde{\mathcal{S}}$ that saves on $a_1$ or there is no $s\in\tilde{\mathcal{S}}$ that saves on $a_1$.

    If there is some $s\in\tilde{\mathcal{S}}$ that shades $b\in\Delta$ and saves on $a_1$, then suppose $\lambda$ is the largest index such that $s_{\lambda}=\ShSa(b,\{a_1\})$.  Also, suppose $\tau$ is the largest index such that $s_{\tau}$ shades $b$. This means
    \begin{equation*}
        S_{\lambda}(b)+T_{\lambda}(b)>S_{\lambda}(a_1)+T_{\lambda}(a_1).
    \end{equation*}
    Since $b\not\in B^*$, we know that $S_{\lambda}(b)<T_{\lambda}(a_1)$ which implies that $T_{\lambda}(b)>S_{\lambda}(a_1)+1$.  Notice that $T_{\lambda}(b)\geq2$ or else the latter inequality would imply that $S_{\lambda}(a_1)<0$.  So $\tau>\lambda$. This means that $T_{\lambda+1}(b)>S_{\lambda+1}(a_1)$.  Notice that for each $i\in[\lambda+1:\tau]$, $s_i$ is of the form: $\ShSa(a_1,W_i)$ where $W_i$ is a subset of $B_i$ of size at most $1$, $\Sh(b)$, or $\ShSa(v,W_i)$ for some $v\in B-\{b\}$ and $W_i\in\{\emptyset,\{a_1\}\}$.  One can use this fact to easily prove by induction on $i$ that $T_{i}(b)>S_{i}(a_1)$ for each $i\in[\lambda+1:\tau]$.   We know $T_{\tau}(b)=1$ and $T_{\tau}(b)>S_{\tau}(a_1)$.  So, $S_{\tau}(a_1)\leq0$.  This along with the fact that $s_{\tau}=\Sh(b)$ means $S_{\tau+1}(a_1)<0$, which is a contradiction.

    If there is no $s\in\tilde{\mathcal{S}}$ that saves on $a_1$, then we know by Inequality~\ref{ineq:sumInDelta} that there must be at least $mt/\epsilon\geq mt$ operations in $\{s_q,\ldots,s_{r-1}\}$ that shade some $b\in\Delta$ but do not save on $a_1$.  Since $S(a_1)=s\leq(m-\epsilon)t$, this implies that $S_{r-1}(a_1)<0$, which is a contradiction.
\end{proof}

Theorem~\ref{thm:CompleteBipartiteLowerBound} tells us that for any $m \in \N$ and $\epsilon>0$ there is an $n\in\N$ such that $ST^{(4)*}(K_{m,n})\geq m+1-\epsilon$; however, as a consequence of Corollary $7$ from~\cite{DKM25} we know that $\chi_{_{DP}}^*(K_{m,n})\leq0.5m+1.75$ for all $m,n\in\N$.  So, Theorem~\ref{thm:CompleteBipartiteLowerBound} shows there is a limitation in using strict type-$4$ degeneracy to bound the fractional DP-chromatic number of complete bipartite graphs. 
\vspace{0.5cm}

{\bf Acknowledgment.}  This paper grew out of a question asked by the second named author at a talk by Anton Bernshteyn on April $8$, $2022$, at the Illinois Institute of Technology\cite{B22}.  The authors would like to thank Hemanshu Kaul for his help and encouragement throughout this project.  They would also like to thank Peter Bradshaw for helpful conversations regarding this paper.

\bibliographystyle{plain}

\bibliography{References}

\appendix 
\section{Appendix}\label{sec:appendix}

    We use the framework described in Remark~\ref{rem:Rigor} to complete the proof of Theorem~\ref{thm:c2r}.

\begin{cl}\label{cl:C2r1}
    Using the notation established in the proof of Theorem~\ref{thm:c2r}, $\mathcal{S}^{(1)}$ has $r$ terms, $\mathcal{S}^{(1)}$ is a restricted $(G,S^{(1)},T^{(1)})$-legal sequence, and $G_{\mathcal{S}^{(1)}}$ is an edgeless graph.
\end{cl}
\begin{proof}
     We prove the following statements for each $i \in [0:r]$ by induction on $i$:
    
    \begin{enumerate}[(1)]
        \item $G_{i+1}=G-\{v_{2\ell}:\ell\in[i]\}$;
        \item $S^{(1)}_{i+1}$ and $T^{(1)}_{i+1}$ satisfy
        $$
            S^{(1)}_{i+1}(v_k)=\left\{\begin{array}{l l}
                1&\text{if }k=1\text{ and }i\neq r\\
                0&\text{if }k=1\text{ and }i=r\\
                1&\text{if }k\in\{2\ell+1:\ell\in[i],\ell\neq r\}\\
                t+1&\text{if }k\text{ is even}\\
                2&\text{otherwise}
            \end{array}\right.
            \text{ and } \;
            T^{(1)}_{i+1}(v_k)=\left\{\begin{array}{l l}
                t&\text{if }k\text{ is odd}\\
                1&\text{otherwise};
            \end{array}\right.
        $$
        \item  if $i \geq 1$, then $(s^{(1)}_1, \dots, s^{(1)}_i)$ is a restricted $(G,S^{(1)},T^{(1)})$-legal sequence.
    \end{enumerate}

    Suppose $i=0$.  Note $G_1=G=G-\emptyset$ which proves statement (1).  We see that
    \begin{align*}
        S^{(1)}_{1}(v_k)&=S^{(1)}(v_k)=\left\{\begin{array}{l l}
                1&\text{if }k=1\\
                t+1&\text{if }k\text{ is even}\\
                2&\text{otherwise}
            \end{array}\right.\\
            &\text{and}\\
            T^{(1)}_{1}(v_k)&=T^{(1)}(v_k)=\left\{\begin{array}{l l}
                t&\text{if }k\text{ is odd}\\
                1&\text{otherwise}
            \end{array}\right.,
    \end{align*} 
    which verifies statement (2).  Since $i=0$, we see that statement (3) is vacuously true.

Now, assume that the three statements of interest are true for all nonnegative integers that are at most $\iota$ where $\iota \in [0:r-1]$.  We first prove statement (1) for $i=\iota+1$.  By the induction hypotheses, we know that $G_{\iota+1}=G-\{v_{2\ell}:\ell\in[\iota]\}$.  Also, we know that $T_{\iota+1}^{(1)}(v_{2\iota+2})=1$.  Since $s^{(1)}_{\iota+1}=\ShSa(v_{2\iota+2},\{v_{2\iota+1}\})$, we see that $G_{\iota+2}=G_{\iota+1}-\{v_{2\iota+2}\}=G-\{v_{2\ell}:\ell\in[\iota+1]\}$ as desired.

    Second, we prove statement (2) for $i=\iota+1$.  By the induction hypotheses,
    \begin{equation*}
        T^{(1)}_{\iota+1}(v_k)=\left\{\begin{array}{l l}
                t&\text{if }k\text{ is odd}\\
                1&\text{otherwise}
            \end{array}\right.
    \end{equation*}
    for each $v_k\in V(G_{\iota+1})$.  Since $s^{(1)}_{\iota+1}=\ShSa(v_{2\iota+2},\{v_{2\iota+1}\})$ and $G_{\iota+2}=G-\{v_{2\ell}:\ell\in[\iota+1]\}$,
    \begin{equation*}
        T^{(1)}_{\iota+2}(v_k)=\left\{\begin{array}{l l}
                t&\text{if }k\text{ is odd}\\
                1&\text{otherwise}
            \end{array}\right.
    \end{equation*}
    for each $v_k\in V(G_{\iota+2})$.  We also know by our induction hypotheses that
    \begin{equation*}
        S^{(1)}_{\iota+1}(v_k)=\left\{\begin{array}{l l}
                1&\text{if }k=1\\
                1&\text{if }k\in\{2\ell+1:\ell\in[\iota]\}\\
                t+1&\text{if }k\text{ is even}\\
                2&\text{otherwise}
            \end{array}\right.
    \end{equation*}
    for each $v_k\in V(G_{\iota+1})$.  Since $s^{(1)}_{\iota+1}=\ShSa(v_{2\iota+2},\{v_{2\iota+1}\})$ and $N_{G_{\iota+1}}(v_{2\iota+2})=\{v_{2\iota+1},v_{2\iota+3}\}$ if $\iota\in[0:r-2]$ while $N_{G_{\iota+1}}(v_{2\iota+2})=\{v_{2\iota+1},v_{1}\}$ if $\iota=r-1$, we see that if $\iota\in[0:r-2]$, $S^{(1)}_{\iota+2}(v_{2\iota+3})=S^{(1)}_{\iota+1}(v_{2\iota+3})-1=2-1=1$ while if $\iota=r-1$, $S^{(1)}_{\iota+2}(v_{1})=S^{(1)}_{\iota+1}(v_{1})-1=1-1=0$.  As a consequence, we see that 
    \begin{equation*}
        S^{(1)}_{\iota+2}(v_k)=\left\{\begin{array}{l l}
                1&\text{if }k=1\text{ and }\iota\neq r-1\\
                0&\text{if }k=1\text{ and }\iota= r-1\\
                1&\text{if }k\in\{2\ell+1:\ell\in[\iota+1,\ell\neq r\}\\
                t+1&\text{if }k\text{ is even}\\
                2&\text{otherwise}
            \end{array}\right.
    \end{equation*}
    as desired.

    Finally, we prove statement (3) for $i=\iota+1$.  By the induction hypotheses, we know that if $\iota\geq1$ then $(s^{(1)}_1,\ldots,s^{(1)}_\iota)$ is a restricted $(G,S^{(1)},T^{(1)})$-legal sequence (in the case where $\iota=0$, the empty sequence is a restricted $(G,S^{(1)},T^{(1)})$-legal sequence).  To show that $(s^{(1)}_1,\ldots,s^{(1)}_{\iota+1})$ is a restricted $(G,S^{(1)},T^{(1)})$-legal sequence, we note that $s^{(1)}_{\iota+1}=\ShSa(v_{2\iota+2},\{v_{2\iota+1}\})$ saves on one vertex, and we now prove that $s^{(1)}_{\iota+1}$ is legal. Note 
    \begin{equation*}
        S^{(1)}_{\iota+1}(v_{2\iota+2})+T^{(1)}_{\iota+1}(v_{2\iota+2})=(t+1)+1=t+2>1+t=S^{(1)}_{\iota+1}(v_{2\iota+1})+T^{(1)}_{\iota+1}(v_{2\iota+1}).
    \end{equation*}
    Also, we see that $S^{(1)}_{\iota+2}(v)\geq0$ for all $v\in V(G_{\iota+2})$.  Hence, $s^{(1)}_{\iota+1}$ is a legal operation and $(s^{(1)}_1,\ldots,s^{(1)}_{\iota+1})$ is a restricted $(G,S^{(1)},T^{(1)})$-legal sequence.  This completes the induction step.  Consequently, $\mathcal{S}^{(1)}$ has $r$-terms, $\mathcal{S}^{(1)}$ is a restricted $(G,S^{(1)},T^{(1)})$-legal sequence, and $G_{\mathcal{S}^{(1)}}=G-\{v_{2\ell}:\ell\in[r]\}$ is an edgeless graph as required.
\end{proof}

\begin{cl}\label{cl:C2r2}
    Using the notation established in the proof of Theorem~\ref{thm:c2r}, if $\mu>1$ then $\mathcal{S}^{(\mu)}$ has $r$  terms, $G_{\mathcal{S}^{(\mu)}}=C_{2r}$, 
    \begin{equation*}
        S_{\mathcal{S}^{(\mu)}}^{(\mu)}(v_i)=\left\{\begin{array}{l l}
            \mu-1&\text{if }i=1\\
            t+1&\text{if }i\text{ is even}\\
            \mu&\text{otherwise},
        \end{array}\right.
        \hspace{0.25in}\text{and}\hspace{0.25in}
        T_{\mathcal{S}^{(\mu)}}^{(\mu)}(v_i)=\left\{\begin{array}{l l}
            t&\text{if }i\text{ is odd}\\
            \mu-1&\text{if }i\text{ is even}.
        \end{array}\right.
    \end{equation*}
\end{cl}
\begin{proof}
    We prove the following statements for each $i \in [0:r]$ by induction on $i$:
    \begin{enumerate}[(1)]
        \item $G^{(\mu)}_{i+1}=G$;
        \item $S^{(\mu)}_{i+1}$ and $T^{(\mu)}_{i+1}$ satisfy
        \begin{align*}
            S^{(\mu)}_{i+1}(v_k)&=\left\{\begin{array}{l l}
                \mu&\text{if }k=1\text{ and }i\neq r\\
                \mu-1&\text{if }k=1\text{ and }i=r\\
                \mu&\text{if }k\in\{2\ell+1:\ell\in[i], \ell \neq r \}\\
                t+1&\text{if }k\text{ is even}\\
                \mu+1&\text{otherwise}
            \end{array}\right.\\
            &\text{and}\\
            T^{(\mu)}_{i+1}(v_k)&=\left\{\begin{array}{l l}
                t&\text{if }k\text{ is odd}\\
                \mu-1&\text{if }k\in\{2\ell:\ell\in[i]\}\\
                \mu&\text{otherwise};
            \end{array}\right.
        \end{align*}
        \item  if $i \geq 1$, then $(s^{(\mu)}_1, \dots, s^{(\mu)}_i)$ is a restricted $(G,S^{(\mu)},T^{(\mu)})$-legal sequence.
    \end{enumerate}

    Suppose $i=0$.  Note $G_1=G$, which proves statement (1).  We see that
    \begin{align*}
        S^{(\mu)}_{1}(v_k)&=S^{(\mu)}(v_k)=\left\{\begin{array}{l l}
                \mu&\text{if }k=1\\
                t+1&\text{if }k\text{ is even}\\
                \mu+1&\text{otherwise}
            \end{array}\right.\\
            &\text{and}\\
            T^{(\mu)}_{1}(v_k)&=T^{(\mu)}(v_k)=\left\{\begin{array}{l l}
                t&\text{if }k\text{ is odd}\\
                \mu&\text{otherwise}
            \end{array}\right.,
    \end{align*} 
    which verifies statement (2).  Since $i=0$, we see that statement (3) is vacuously true.

    Now we assume that the three statements of interest are true for all nonnegative integers that are at most $\iota$ where $\iota\in[0:r-1]$.  We first prove statement (1) for $i=\iota+1$.  By the induction hypothesis, we know that $G_{\iota+1}=G$.    Also, since $s^{(\mu)}_{\iota+1}=\ShSa(v_{2\iota+2},\{v_{2\iota+1}\})$ and $T_{\iota+1}^{(\mu)}(v_{2\iota+2}) = \mu >1$, $v_{2\iota+2}$ is not removed from $G$ by $s^{(1)}_{\iota+1}$.  Consequently $G_{\iota+2}=G_{\iota+1}=G$ as desired.

    Second, we prove statement (2) for $i=\iota+1$.  By the induction hypothesis,
    \begin{equation*}
        T^{(\mu)}_{\iota+1}(v_k)=\left\{\begin{array}{l l}
                t&\text{if }k\text{ is odd}\\
                \mu-1&\text{if }k\in\{2\ell:\ell\in[\iota]\}\\
                \mu&\text{otherwise}
            \end{array}\right.
    \end{equation*}
    for each $v_k\in V(G_{\iota+1})$.  Since $s^{(\mu)}_{\iota+1}=\ShSa(v_{2\iota+2},\{v_{2\iota+1}\})$
    \begin{equation*}
        T^{(\mu)}_{\iota+2}(v_k)=\left\{\begin{array}{l l}
                t&\text{if }k\text{ is odd}\\
                \mu-1&\text{if }k\in\{2\ell:\ell\in[\iota+1]\}\\
                \mu&\text{otherwise}
            \end{array}\right.
    \end{equation*}
    for each $v_k\in V(G_{\iota+2})$.  We also know by our induction hypothesis that
    \begin{equation*}
        S^{(\mu)}_{\iota+1}(v_k)=\left\{\begin{array}{l l}
                \mu&\text{if }k = 1\\
                \mu&\text{if }k\in\{2\ell+1:\ell\in[\iota]\}\\
                t+1&\text{if }k\text{ is even}\\
                \mu+1&\text{otherwise}
            \end{array}\right.
    \end{equation*}
    for each $v_k\in V(G_{\iota+1})$.  Since $s^{(\mu)}_{\iota+1}=\ShSa(v_{2\iota+2},\{v_{2\iota+1}\})$ and since $N_{G_{\iota+1}}(v_{2\iota+2})=\{v_{2\iota+1},v_{2\iota+3}\}$ if $\iota\in[r-2]$ while $N_{G_{\iota+1}}(v_{2\iota+2})=\{v_{2\iota+1},v_{1}\}$ if $\iota=r-1$, we see that if $\iota\in[0:r-2]$, $S^{(\mu)}_{\iota+2}(v_{2\iota+3})=S^{(\mu)}_{\iota+1}(v_{2\iota+3})-1=(\mu+1)-1=\mu$ while if $\iota=r-1$, $S^{(\mu)}_{\iota+2}(v_{1})=S^{(\mu)}_{\iota+1}(v_{1})-1=\mu-1$.  As a consequence, 
    \begin{equation*}
        S^{(\mu)}_{\iota+2}(v_k)=\left\{\begin{array}{l l}
                \mu&\text{if }k=1\text{ and }\iota\neq r-1\\
                \mu-1&\text{if }k=1\text{ and }\iota=r-1\\
                \mu&\text{if }k\in\{2\ell+1:\ell\in[\iota+1], \ell \neq r\}\\
                t+1&\text{if }k\text{ is even}\\
                \mu+1&\text{otherwise}
            \end{array}\right.
    \end{equation*}
    as desired.

    Finally, we prove statement (3) for $i=\iota+1$.  By the induction hypothesis, we know that if $\iota\geq1$ then $(s^{(\mu)}_1,\ldots,s^{(\mu)}_\iota)$ is a restricted $(G,S^{(\mu)},T^{(\mu)})$-legal sequence (in the case where $\iota=0$, the empty sequence is a restricted $(G,S^{(\mu)},T^{(\mu)})$-legal sequence).  To show that $(s^{(\mu)}_1,\ldots,s^{(\mu)}_{\iota+1})$ is a restricted $(G,S^{(\mu)},T^{(\mu)})$-legal sequence, we note that $s^{(\mu)}_{\iota+1}=\ShSa(v_{2\iota+2},\{v_{2\iota+1}\})$ saves on one vertex, and we now prove that $s^{(\mu)}_{\iota+1}$ is legal.  Note
    \begin{equation*}
        S^{(\mu)}_{\iota+1}(v_{2\iota+2})+T^{(\mu)}_{\iota+1}(v_{2\iota+2})=(t+1)+\mu=t+\mu+1>\mu+t=S^{(\mu)}_{\iota+1}(v_{2\iota+1})+T^{(\mu)}_{\iota+1}(v_{2\iota+1}).
    \end{equation*}
    Also, we see that $S^{(\mu)}_{\iota+2}(v_k)\geq0$ for all $v_k\in V(G_{\iota+2})$.  Hence, $s^{(\mu)}_{\iota+1}$ is a legal operation and $(s^{(\mu)}_1,\ldots,s^{(\mu)}_{\iota+1})$ is a $(G,S^{(\mu)},T^{(\mu)})$-legal sequence.  This completes the induction step.  Consequently, $\mathcal{S}^{(\mu)}$ has $r$-terms, $\mathcal{S}^{(\mu)}$ is a restricted $(G,S^{(\mu)},T^{(\mu)})$-legal sequence, $G_{\mathcal{S}^{(\mu)}}=G$,
    \begin{align*}
        S_{\mathcal{S}^{(\mu)}}^{(\mu)}(v_i)&=S^{(\mu-1)}(v_i)=\left\{\begin{array}{l l}
            \mu-1&\text{if }i=1\\
            t+1&\text{if }i\text{ is even}\\
            \mu&\text{otherwise},
        \end{array}\right.\\
        &\text{and}\\
        T_{\mathcal{S}^{(\mu)}}^{(\mu)}(v_i)&=T^{(\mu-1)}(v_i)=\left\{\begin{array}{l l}
            t&\text{if }i\text{ is odd}\\
            \mu-1&\text{if }i\text{ is even},
        \end{array}\right.
    \end{align*}
    as claimed.
\end{proof}

\end{document}